\documentclass[11pt]{article}
\usepackage[T1]{fontenc}
\usepackage[a4paper,margin=1in]{geometry}
\usepackage{amsmath,amssymb,amsthm,mathtools}
\usepackage{enumitem}
\usepackage{hyperref}

\newtheorem{theorem}{Theorem}
\newtheorem{lemma}{Lemma}
\newtheorem{corollary}{Corollary}
\newtheorem{remark}{Remark}

\newcommand{\R}{\mathbb R}
\newcommand{\N}{\mathbb N}
\newcommand{\supp}{\operatorname{supp}}
\newcommand{\Lie}{\mathcal L}

\usepackage{lineno}
\usepackage{amsmath}

\makeatletter
\newcommand*\patchAmsMathEnvironmentForLineno[1]{%
	\expandafter\let\csname old#1\expandafter\endcsname\csname #1\endcsname
	\expandafter\let\csname oldend#1\expandafter\endcsname\csname end#1\endcsname
	\renewenvironment{#1}%
	{\linenomath\csname old#1\endcsname}%
	{\csname oldend#1\endcsname\endlinenomath}%
}
\patchAmsMathEnvironmentForLineno{equation}
\patchAmsMathEnvironmentForLineno{align}
\makeatother

\newif\iflncs\lncsfalse
\newif\ifmai\maifalse
\newif\ifrunningex\runningexfalse

\usepackage{rotating}
\usepackage[english]{babel}
\usepackage{bbm}
\usepackage{amsthm}
\usepackage[T1]{fontenc}
\usepackage{algorithm}
\usepackage{svg}
\usepackage{algpseudocode}
\usepackage{listings}
\usepackage{ifthen}
\usepackage{geometry}

\usepackage{chngcntr}
\usepackage{apptools}
\usepackage{xcolor}

\usepackage{mathtools}
\usepackage{booktabs}
\usepackage{multirow}
\usepackage{amsthm}
\usepackage{geometry}

\ifmai
\theoremstyle{plain}
\renewenvironment{proof}{\noindent{\sc Proof.}~}{ \hfill $\Box$ \vskip 2mm }
\newenvironment{proof_of}[1]{\noindent{\sc Proof of #1.}~}{ \hfill $\Box$ \vskip .2mm}
\iflncs{}\else \fi

\iflncs{}\else \fi

\iflncs{}\else \newtheorem{theorem}{Theorem}\fi

\iflncs{}\else \newtheorem{corollary}{Corollary} \fi

\iflncs{}\else \newtheorem{lemma}{Lemma}\fi

\newtheorem{remark}{Remark}[section]
\fi
\renewenvironment{proof}{\noindent{\sc Proof.}~}{ \hfill $\Box$ \vskip 2mm }
\newenvironment{proof_of}[1]{\noindent{\sc Proof of #1.}~}{ \hfill $\Box$ \vskip .2mm}

\usepackage{amssymb}
\usepackage{graphicx}
\usepackage{graphics}
\usepackage{pgf}
\usepackage{tikz}
\usepackage{wrapfig}

\newcommand{\mik}[1]{\marginpar{ \textbf{Mic:} {\footnotesize #1}}}
\newcommand{\mic}[1]{}

\newcommand{\es}{\emptyset}

\newcommand{\reals}{{\mathbb{R}}}

\newcommand{\ode}{{\sc ode}}

\newcommand{\lie}{\mathcal{L}}
\newcommand{\zz}{z}

\newbox\mybox
\newdimen\myboxwidth

\newcommand{\rt}{{\sc ckr}}

\newcommand{\lui}[1]{}
\usepackage{physics}
\usepackage{tikz}
\usepackage{mathdots}
\usepackage{cancel}
\usepackage{color}
\usepackage{array}
\usepackage{multirow}
\usepackage{amssymb}
\usepackage{gensymb}
\usepackage{tabularx}
\usepackage{extarrows}
\usepackage{booktabs}
\usetikzlibrary{fadings}
\usetikzlibrary{patterns}
\usetikzlibrary{shadows.blur}
\usetikzlibrary{shapes}
\usepackage{hhline}

\def\BibTeX{{\rm B\kern-.05em{\sc i\kern-.025em b}\kern-.08em
	T\kern-.1667em\lower.7ex\hbox{E}\kern-.125emX}}

\usepackage{chngcntr}
\usepackage{apptools}

\usepackage{listings}
\usepackage[numbers]{natbib}

\begin{document}

\title{New convergence results for Carleman linearization}
\author{Michele Boreale and Luisa Collodi\\
Universit\`a    di Firenze}

\date{}



\maketitle

\begin{abstract}
We prove new error bounds for finite Carleman truncations of polynomial
ordinary differential equations. The analysis works directly in the original
monomial basis and for selected observables, such as state coordinates. Using a
Dyson--Duhamel expansion, we separate the degree-preserving linear part from the
degree-raising nonlinear part and track how truncation errors can propagate back
to the observable. The resulting bounds are degree-aware and retain
logarithmic-norm information from the original linear dynamics. We obtain
explicit finite-degree estimates and geometric convergence over certified time
horizons. Comparisons with existing bounds, in particular those of
Forets--Pouly, are given on the Stuart--Landau and Van der Pol systems.
\end{abstract}


\ifmai
\section{Introduction}
Reachability analysis is a central problem in the formal verification of continuous
and hybrid systems. Given a dynamical system and a set of initial states, the goal is
to compute a sound enclosure of all states that can be reached over a prescribed time
horizon. This problem underlies safety verification for cyber-physical systems and has
motivated a large body of tools and methods, including SpaceEx~\cite{Frehse2011},
Flow*~\cite{Chen2013}, CORA~\cite{Althoff2015}, and JuliaReach~\cite{Bogomolov2019}.
For nonlinear ordinary differential equations, most practical algorithms rely on local
approximations: Taylor models, local linearization, zonotopes, support functions, or
combinations thereof: see  the
survey~\cite{AlthoffFrehseGirard2021} and references thereof. For nonlinear ordinary differential
equations, however, reachability remains difficult. Most successful methods build
flowpipes in a piecewise way, by repeatedly constructing local approximations in
time or in space. This locality is often unavoidable, but it can also lead to
accumulated over-approximation and frequent recomputation.

A classical way to simplify nonlinear dynamics is to replace them by linear
dynamics. The standard first-order linearization does this around a fixed point,
and is therefore intrinsically local. Carleman linearization offers a different
possibility. Instead of linearizing the vector field around a point, it lifts the
system to a space of monomials, where the evolution becomes linear but
infinite-dimensional. A finite truncation of this lifted system then provides a
linear approximation of the original nonlinear dynamics. This construction is
particularly attractive for formal methods: if the nonlinear system can be
approximated by a finite linear system with a certified error bound, then one may
combine nonlinear verification with mature linear reachability techniques.

This is the motivation behind our previous work on Carleman--Krylov Reachability
(CKR)~\cite{BorealeCollodi2025}. In CKR, Carleman linearization is combined with
Krylov projection to obtain a small linear system, independent of the particular
initial state, that can be used to propagate reachsets represented as zonotopes.
The approximation error is then compensated by inflating the propagated sets. In
that setting, the quality of the Carleman truncation bound is not a merely
analytic issue: it directly affects the size of the bloating terms and therefore
the precision of the computed flowpipe.

Explicit error bounds for Carleman truncations have been studied before. Forets
and Pouly~\cite{ForetsPouly2017} gave computable bounds for polynomial systems,
including estimates based on a priori trajectory bounds and estimates obtained
after reducing polynomial systems to quadratic form. Amini, Zheng, Sun and
Motee~\cite{AminiZhengSunMotee2025} developed a broader finite-section theory
for analytic systems, covering also non-polynomial vector fields under suitable
coefficient-decay assumptions. These works provide important convergence
guarantees, but their bounds can be conservative when used as numerical
certificates for polynomial systems arising in reachability. In particular,
coarse norm estimates and lifted representations may hide useful structure of the
original system.

The present paper revisits the error analysis of truncated Carleman
linearization from a viewpoint closer to reachability. We focus on polynomial
systems and on observables of interest, such as the individual state coordinates
that define the reachable set. Rather than passing immediately to a quadratic
lift, we work directly with monomials in the original variables. This allows the
analysis to keep track of how errors generated in high-degree monomials can
propagate back to a low-degree observable. The key observation is simple: the
linear part of the vector field preserves monomial degree, whereas the nonlinear
part can only increase it. Hence a truncation error lying beyond degree \(D\)
cannot affect a coordinate observable unless enough nonlinear degree-raising
steps have occurred.

Technically, this idea is implemented through a Dyson--Duhamel expansion of the
matrix exponential associated with the truncated Carleman system. This expansion
separates the contribution of the linear part from the nonlinear degree-raising
part. It also makes it possible to count the number of nonlinear interactions
needed to connect the discarded tail with the observable. The resulting estimate
is degree-aware: instead of bounding every nonlinear interaction by the worst
degree \(D\), it follows the degrees that can actually be reached along each
path. This gives sharper finite-dimensional bounds.

A second feature of the proof is the use of logarithmic norms. When the linear
part of the system is contractive, this contractivity should improve the error
bound. Some previous norm-based estimates do not fully exploit this information.
Our analysis keeps the linear and nonlinear effects separate, so that a negative
logarithmic norm of the linear dynamics enters the bound quantitatively. This is
especially relevant in stable or weakly nonlinear regimes, which are common in
verification benchmarks.

The paper compares the resulting bounds with those of Forets--Pouly and
Amini--Zheng--Sun--Motee on the Stuart--Landau and Van der Pol systems. These
examples show both the strengths and the limitations of the approach. The point is
not that quadratic lifting is inherently impractical: in an efficient
implementation, many redundant monomials can be identified. Rather, the advantage
of the present analysis is conceptual and structural. It works in the original
variables, is tailored to the observable being verified, uses degree information
explicitly, and preserves logarithmic-norm information from the original system.
These features make the bounds well suited for integration into
Carleman-based reachability algorithms such as CKR.

\paragraph*{Related  Work}
There exists a vast literature  on the linearization of nonlinear systems. In particular,  techniques based on Carleman
embedding \cite{Bell,Kowa}  have recently received a renewed attention. Most related to our work and  motivations, Jungers and Tabuada \cite{JT19} have recently proposed a technique for global approximation of nonlinear \ode s by linear \ode s, based on \emph{polyflows}. These are systems that are exactly linearizable  via a change of variables.   The technique in \cite{JT19}  is based on building polyflows that approximate the original system, using as a basis the Lie derivatives up to some order $N$ and as $N\rightarrow +\infty$  the approximation of \cite{JT19} becomes exact. Note that this is an asymptotic result that does not easily yield concrete bounds for a fixed $N$.
Systems that are exactly linearizable via  polynomial changes of variables are the subject of \cite{SankaLin1,SankaLin2,Bor18Full};  see also \cite{B1,SOFSEM18,B3}.

In \cite{Bor18Full}    we have considered
Carleman embedding and Krylov-based   approximations,     essentially from a   local point of view. Here, we provide   novel   analyses of both local and   global errors, and leverage them in \rt, a new reachability algorithm.
General error bounds for the truncated Carleman linearization have been recently   considered in \cite{Forets2,Amini}. The time interval of validity of these bounds is quite small, and they appear to be  in practice more conservative than our; cf. experiments and comparison in Section \ref{sec:exp}.
In \cite{Forets1},  efficient reachability for weakly nonlinear, dissipative systems relying on Carleman linearization is presented.  We do not have such restrictions, but it is well possible that the approach in \cite{Forets1} is more effective than ours limited to that class of systems. In the conference version \cite{BC22} an earlier, less stable version of CKR, based on polytopes,  is considered.

In the field of  reachability  for continuous and hybrid systems, state-of-the-art tools like Flow$*$ \cite{Flow} and {\sc cora} \cite{CORA} employ a mix of    approximations techniques  \cite{TaylorModels,TaylorFlow,Flow,CORA,Althoff,ReachReview}.
In particular, Flow$*$ \cite{TaylorFlow,Flow} is based on Taylor models, while  {\sc cora} mainly relies on    linearization of the  \ode\ equations \cite{Althoff,ReachReview}.
These tools are quite effective at building piecewise overapproximations of reachsets, as explained above.
Very recently, Koopman and Carleman approximations have also been applied to reachability analysis   in \cite{Bak22}. In order to achieve greater accuracy, they employ randomly generated Fourier features as observables. Dimension reduction and on-the-fly generation of the model are    not considered. However, they consider discretized versions of the models, and focus on safety tasks, which makes a direct comparison with our experimental results   not possible.


\fi

\section{Introduction}\label{sec:intro}

Reachability analysis is a central problem in the formal verification of
continuous and hybrid systems. Given a dynamical system and a set of initial
states, the goal is to compute a sound enclosure of all states that may be
visited over a prescribed time horizon. This problem is fundamental in the
verification of cyber-physical and safety-critical systems, and has motivated a
large body of methods and tools, including SpaceEx~\cite{Frehse2011},
Flow$*$~\cite{Flow}, CORA~\cite{CORA}, JuliaReach~\cite{Bogomolov2019}, and
zonotope-based approaches~\cite{zono}; see also the survey
\cite{ReachReview}. For nonlinear ordinary differential equations, the problem
remains particularly challenging. Most practical algorithms construct flowpipes
piecewise, relying on Taylor models, local linearization, conservative
polynomialization, set propagation, or combinations of these techniques
\cite{TaylorModels,TaylorFlow,Flow,Althoff,CORA,ReachReview}. The resulting   approximation is \emph{local},
and must be updated
along the reachability steps, which may be computationally costly  and   contribute to error accumulation.

A classical way to simplify nonlinear dynamics is to replace them by linear
dynamics. The standard first-order linearization does so around a fixed point or
around the current reachset, and is therefore local in nature.
\emph{Carleman
linearization}~\cite{Carleman1932,Bell,Kowa} offers a different possibility.
Starting from a polynomial ordinary differential equation, it lifts the system to
a space of monomials, in which the dynamics become linear but
infinite-dimensional. Truncating the lifted system at a finite degree then gives
a finite-dimensional linear approximation of the original nonlinear system. This
is attractive from the point of view of formal verification: linear reachability
techniques are well developed, and a fixed autonomous Carleman truncation can be
constructed once for a chosen degree and then reused across reachability steps.
What remains step-dependent is not the truncated Carleman matrix itself, but the
lifted reachset and the local truncation-error enclosure.
This observation is one of the motivations behind our previous work on
Carleman--Krylov Reachability (CKR)~\cite{full,BC22}, a zonotope-based reachability
algorithm. In CKR, Carleman linearization is combined with Krylov
projection~\cite{Saad} to obtain a small linear system that approximates the
observable dynamics of the original nonlinear system.

A related motivation comes from quantum algorithms for differential equations.
Quantum linear-system and linear-ODE algorithms~\cite{HarrowHassidimLloyd2009,
BerryChildsOstranderWang2017} provide a route to solving certain large linear
problems in a quantum state representation. Carleman linearization has therefore
become a natural preprocessing step for quantum algorithms for nonlinear
dynamics: one first embeds the nonlinear ODE into a higher-dimensional linear
system, truncates it, and then applies a quantum linear-algebraic subroutine.
This idea underlies the quantum algorithm of Liu et al.~\cite{LiuEtAl2021} for
dissipative nonlinear differential equations, and has since been refined and
extended in several directions~\cite{Krovi2023,CostaSchleichMoralesBerry2025,
WuWangLi2025}. A common feature of much of this line of work is that efficiency
and convergence are obtained under some stability, dissipativity, or
contractivity hypotheses.
Recent work
has begun to relax these requirements~\cite{WuWangLi2025},
but still e.g. requiring nonpositive real parts of the eigenvalues.
Although quantum computation is not the  focus of the present paper, this
line of work reinforces the same basic point: one seeks  certified convergence and truncation-error
estimates  in a general setting not
requiring    dissipativity or contractivity assumptions. 

The present paper revisits the error analysis of truncated Carleman
linearization with these motivations in mind.
Compared to the framework of CKR~\cite{full}, we leave Krylov projection aside and
focus on the problem: whether, and with what explicit error bound, the degree-\(D\)
Carleman truncation converges to the original nonlinear dynamics. We also focus on observables,
 such as state coordinates. This is the setting most directly
relevant to reachability, where one ultimately needs certified
enclosures for the components of the state vector.

Convergence and error bounds for Carleman linearization in a general setting are
already available: most notably in the work of Forets and Pouly~\cite{Forets2}
for polynomial systems; and for possibly non-polynomial analytical systems in the work by
Amini, Sun and Motee~\cite{Amini} and Amini, Zheng, Sun and
Motee~\cite{AminiZhengSunMotee2025}. Our aim is not to replace these results by
a more general theory, but to sharpen the analysis in the polynomial case, in a
way that is useful for reachability and, more generally, for applications where
finite Carleman truncations must be certified.
Compared with previous bounds, the main point of the present analysis is that it
keeps more of the original polynomial structure visible. Rather than first
reducing a general polynomial system to a quadratic one as in~\cite{Forets2}, we
work directly with monomials in the original variables. This allows us to exploit
the degree structure of the Carleman matrix: the part induced by the linear
vector field does not mix total monomial degrees, whereas the part induced by the
nonlinear terms only moves information towards higher degrees.

Technically, our proof keeps the linear and non linear parts  separate through a
\emph{Dyson--Duhamel expansion} of the matrix exponential
\cite{Dyson,Pazy1983}.
This also allows us to control the number of
nonlinear interactions  across which information can pass from
the observable to the truncated tail. 
This yields sharper
finite-degree estimates while keeping the constants expressed in terms of the
original vector field.
In particular, keeping the effects of the linear and nonlinear parts separate   enables an
effective use of logarithmic norms. If the linear part of the original system is
contractive, in the sense of having \emph{negative} logarithmic norm (\(\mu_\infty(\cdot)\)), this should
improve the truncation bound: roughly,
\[
  \|e^{At}\|\le e^{\mu_\infty(A)t}\,.
\]
Coarser norm estimates may
lose this information, especially after lifting or after replacing the original
dynamics by a larger quadratic system.
In particular, we obtain explicit finite-degree
error bounds and derive exponential convergence of the degree-\(D\) truncation,
uniformly over a certified time horizon, with a geometric factor that reflects
the logarithmic norm of the original linear dynamics.

To sum up, when compared to previous general purpose error analyses~\cite{Forets2,Amini,AminiZhengSunMotee2025},
 the advantage of our approach
 is that it  uses degree
information explicitly and  separates the linear and nonlinear mechanisms,
which in turn  preserves the  logarithmic-norm information from the
original system.
This
difference can be clearly appreciated  in the examples of Section~\ref{sec:exp}. We
consider two qualitatively different systems, one contractive  (Stuart--Landau) and one with
limit-cycle dynamics (Van der Pol):
in both cases, our estimates are significantly  less
conservative than the  bounds of~\cite{Forets2} and  lead to better geometric factors and larger certified
time horizons.

\paragraph*{Related Work.}
Reachability analysis for nonlinear continuous and hybrid systems has been
studied through several complementary techniques. Taylor-model methods are the
basis of Flow$*$~\cite{TaylorModels,TaylorFlow,Flow,chenT}; they construct
validated polynomial flowpipe segments and associated remainder bounds. CORA
uses a broad collection of set-propagation techniques, including
linearization-based methods and conservative polynomialization
\cite{Althoff,CORA,ReachReview}. These approaches are quite effective in
practice, but their approximations are generally local to the current flowpipe
segment. In particular, Flow$*$ reuses the vector field and Taylor-model
arithmetic machinery, but the actual Taylor-model enclosure, remainder, order,
and step size are updated along the computation. Standard Jacobian-based
linearization is even more explicitly step-local, since the linear model and its
remainder depend on the current expansion point or set.

Carleman linearization has a long history in the analysis of nonlinear
differential equations~\cite{Carleman1932,Bell,Kowa}. More recently, it has
received renewed attention in verification, control, reachability, and quantum
algorithms. Forets and Pouly~\cite{Forets2} developed explicit error bounds for
polynomial systems. Forets and Schilling~\cite{Forets1} used Carleman
linearization for reachability of weakly nonlinear dissipative systems. Amini,
Sun and Motee~\cite{Amini}, and later Amini, Zheng, Sun and
Motee~\cite{AminiZhengSunMotee2025}, studied finite-section convergence for
broader classes of analytic systems. Their setting is more general than
\cite{Forets2}'s and ours, but, when specialized to the polynomial setting, Amini
et al.'s bounds are significantly more conservative in practice. Our
analysis is more specialized and aims at sharper constants for polynomial
systems and coordinate observables.
The comparison with the trajectory-dependent backward estimate
of~\cite{Forets2} is more nuanced, since that estimate also exploits
contractivity when present; nevertheless, our analysis does so directly in the original
monomial basis and for the chosen observable, resulting in a clearer exploitation of contractivity; see experiments and further discussion
in Section \ref{sec:exp}.
The approach of~\cite{Forets1} is tailored
to weakly nonlinear dissipative systems; we do not impose such a restriction,
although that approach may be more effective within its target class.

\ifmai
Our bounds should be compared with the explicit estimates of Forets and
Pouly~\cite{Forets2} and with the finite-section convergence results of Amini,
Sun and Motee~\cite{Amini} and Amini, Zheng, Sun and
Motee~\cite{AminiZhengSunMotee2025}. The former provide computable bounds for
polynomial systems, including estimates obtained after quadratic lifting and
bounds based on a priori trajectory information. The latter develop a broader
theory for analytic, possibly non-polynomial, systems under coefficient-decay
assumptions. These works give important convergence guarantees. Our contribution
is complementary: by restricting attention to polynomial systems and tracking
degree propagation more closely, we obtain bounds that are tailored to
coordinate observables and to reachability-oriented error accounting. This
difference is reflected in the examples of Section~\ref{sec:exp}, where we
consider two qualitatively different systems: a Stuart--Landau system whose
linear part has negative logarithmic norm, and a Van der Pol oscillator
representative of self-sustained limit-cycle dynamics. Our estimates are less
conservative than the norm-based/power-series bounds of~\cite{Forets2} and, in
the examples considered, lead to better geometric factors and larger certified
time horizons. The comparison with the trajectory-dependent backward estimate
of~\cite{Forets2} is more nuanced, since that estimate also exploits
contractivity; nevertheless, our analysis does so directly in the original
monomial basis and for the chosen observable.

The resulting picture is nuanced. The advantage of the present analysis is not
simply that it avoids quadratic lifting~\cite{Forets2}. In an efficient
implementation, products can be identified up to commutativity, and some of the
dimension growth of an ordered Kronecker representation can be avoided. Rather,
the advantage is structural. The proof works in the original variables, is
observable-specific, separates the linear and nonlinear mechanisms, uses degree
information explicitly, and preserves logarithmic-norm information from the
original system. In perspective, these properties make the estimates well suited
for integration into Carleman-based reachability algorithms such as CKR, and
also relevant to other settings where certified Carleman truncations are used as
a preprocessing step for linear methods.
\fi

A recent and rapidly developing use of Carleman linearization appears in quantum
algorithms for nonlinear differential equations. Quantum algorithms for linear
systems and linear ODEs~\cite{HarrowHassidimLloyd2009,
BerryChildsOstranderWang2017} provide the linear-algebraic background for this
line of work. Liu et al.~\cite{LiuEtAl2021} used Carleman linearization to
construct an efficient quantum algorithm for a class of dissipative quadratic
nonlinear ODEs. Krovi~\cite{Krovi2023} gave improved quantum algorithms for
linear and nonlinear ODEs, while Costa et al.~\cite{CostaSchleichMoralesBerry2025}
further refined Carleman-based quantum algorithms using higher-order methods and
rescaling. The recent work of Wu, Wang and Li~\cite{WuWangLi2025} revisits
Carleman linearization in quantum algorithms beyond the traditional dissipative
condition. These works have different goals from ours, but they rely on the same
central issue: a nonlinear system is useful to a linear solver only if the
Carleman truncation can be controlled by explicit error estimates.


Koopman-based
and Carleman-based approximations have also been used in recent reachability
methods. For instance, Bak et al.~\cite{Bak22} combine Koopman linearization
with random Fourier feature observables and polynomial zonotope refinement, with a focus on safety verification. Our previous work~\cite{full,BC22} introduced CKR as a reachability algorithm based
on Carleman linearization, Krylov projection, and zonotope propagation.
In
CKR,   error compensation   is local:
it is based on Taylor-remainder estimates for the chosen reduced observable
dynamics over each step, rather than on a uniform convergence analysis of the
Carleman truncation as the degree tends to infinity.
The present paper should be read as a refinement of the theoretical
error-analysis component underlying CKR.
We do not propose a new reachability
data structure here; instead, we provide sharper Carleman truncation estimates
that can be used to reduce the error inflation needed in CKR-style algorithms.

\paragraph*{Structure of the paper} We introduce Carleman linearization in Section \ref{sec:prel}. The main results and an overview of their proofs are presented in Section \ref{sec:conv}. Section \ref{sec:exp} is devoted to numerical experiments and result  comparison  with related work. Detailed proofs of the results are in Section \ref{sec:proofs}. Concluding remarks and future work are discussed in Section \ref{sec:conclusion}.

\newcommand{\unitmon}{\epsilon}
\section{Carleman Linearization}\label{sec:prel}
%
We   introduce Carleman  linearization (or embedding) in a polynomial setting. Our presentation follows mostly  
\cite{Kowa}.
For $x=(x_1,...,x_n)^T$ a vector of state variables, we consider a   system  of \ode s
	\begin{equation}
	\dot x(t)  = f(x(t))\label{eq:system}
\end{equation}
\noindent
where $f=(f_1,...,f_n)^T$ is a  polynomial vector field, that is each $f_i:\reals^n\to \reals$ is a multivariate polynomial.
Fix $x_0\in \reals^n$: we let $x(t;x_0)$ be the unique solution
of the \ode\ system with the initial condition $x(0)=x_0$. The  solution exists, is unique and is real analytic  by virtue of  the Picard-Lindel\"{o}f theorem \cite{Chicone}. For a  real  analytic function $g$ defined on some open subset of $\reals^n$ that includes  the trajectories $x(t;x_0)$,  one is typically     interested in studying the \emph{observable} of the system
\eqref{eq:system} via $g$, that is the function
\[ g\circ x(t;x_0)=g(x(t;x_0)).\]
We will stick to the case where $g$ is a multivariate polynomial.
For $\beta\in \N^n$, we will let $x^\beta:=\Pi_{i=1}^n x_i^{\beta_i}$ be a monomial of \emph{total degree} $\deg(x^\beta):=|\beta| =\sum_{i=1}^n\beta^i$. As a function,  $x^\beta:\reals^n\to \reals$. Now
fix a sequence
\[\alpha_1,\alpha_2,\alpha_3,...\]
of all nonzero {monomials},  ordered in such a way that $i<j$ implies $\deg(\alpha_i)\leq \deg(\alpha_j)$. 
Fix a \emph{truncation order} $M\geq 1$ and let
\[\alpha:=
	(\alpha_1,...,\alpha_M)^T\]
be the resulting monomial basis.
The truncation order $M$ is chosen large enough that  there is a (unique) vector
	${ v}=(\lambda_1,...,\lambda_M)^T\in \reals^M$   with {
		\begin{equation}\label{eq:g}
			g   =   \sum_{i=1}^M \lambda_i \alpha_i \;=\;{ v^T} \alpha\,.
	\end{equation}}\noindent
For a function $h:\reals^n\to \reals$ the \emph{Lie derivative} of $h$ w.r.t. $f$  is $\Lie_f(h):= \langle \nabla h,f\rangle=\sum_{j=1}^n\frac{\partial h}{\partial x_j}  \cdot f_j$, while $\Lie_f^{(k)}(h)$ is the $k$-th Lie derivative, defined inductively by $\Lie_f^{(k+1)}(h):=\Lie_f(\Lie_f^{(k)}(h))$. We shall omit the subscript ${}_f$ whenever it is understood from the context.
Since the Lie derivative of a monomial is a polynomial, for each $i=1,...,M$ there are unique coefficients $a_{ij}$'s such that
\[\lie(\alpha_i)  =   \sum_{j\geq 1}a_{ij}\alpha_j\,.\]
	%
For each $i$, only finitely many coefficients $a_{ij}$ here are nonzero.
	Let $A$ denote the $M\times M$ matrix of the coefficients $a_{ij}$  for $1\leq i,j\leq M$, and $B$ be the $M\times k$ matrix of possibly nonzero elements
	$b_{i,j}=a_{i,M+j}$; that is, $k$ is chosen large enough to ensure that, for $1\leq i\leq M$,
	we have $a_{ij}=0$ for   each $j>M+k$.  We  let
	\[\psi := (\alpha_{M+1} ,...,\alpha_{M+k} )^T.\]
	The \emph{Carleman linearization}  (or embedding) of  \eqref{eq:system} is  given by the following linear  system   in the variables
	$\zz=(z_1,...,z_M)^T$ and initial condition  {
		\begin{align}\label{eq:system2}
			\dot \zz  &  =   A    \zz + B  \cdot \psi(x(t;x_0))&&&
			\zz(0)   & =     \alpha(x_0) \;=:\;z_0\,.
	\end{align}}\noindent
The following result is an almost immediate consequence of the existence and uniqueness of the solution of \ode s (Picard-Lindel\"{o}f). For a detailed proof, see \cite[Th.3]{Bor18Full}.
	
\begin{theorem}[Carleman linearization]\label{th:section}
		$\alpha(x(t;x_0))$
		is the unique solution of the system
		\eqref{eq:system2}. 
	\end{theorem}\noindent
Note that we cannot \emph{explicitly} build the system
	\eqref{eq:system2}, as the function  $\psi(x(t;x_0))$ is in general not available. This leads us to consider an approximation where we   neglect the ``overflow'' term $w(t):=B\cdot\psi(x(t;x_0))$, the \emph{truncated} Carleman linearization of dimension $M$
	\begin{align}\label{eq:carltrunc}
		\dot z &=Az  &&  z(0) =z_0(=\alpha(x_0))\,.
	\end{align}	
In what follows, \emph{we will let $z(t;z_0)=e^{At}\alpha(x_0)$ denote the unique solution of} of the truncated linear system \eqref{eq:carltrunc}. Note that $g(x(t;z_0))=v^T\alpha(x(t;x_0))$. So it is natural to consider the \emph{approximate observable} as $v^T z(t;z_0)$ and the corresponding \emph{error}, for $t$ in the interval of definition of $x(t;x_0)$:
\begin{align}\label{eq:error}
\epsilon(t;x_0)&:= v^T \alpha(x(t;x_0))-v^T z(t;z_0)\,.
\end{align}
Note that   $z(t;z_0), \epsilon(t;x_0)$  implicitly  depend  on the truncation order $M$: we may write $z_M(t;z_0)$, $\epsilon_M(t;x_0)$ to make this dependence explicit in the notation. In the special  case which most interests us, when $\alpha$ consists of all   nonzero monomials of given total degree $\leq D$,  we will employ the notation $z_{M(D)}(t;z_0), \epsilon_{M(D)}(t;x_0)$, or more simply  $z_{D}(t;z_0), \epsilon_{D}(t;x_0)$. We analyze the error $\epsilon_{D}(t;x_0)$ in the next section.

\section{Statement and Overview of the Main Results}\label{sec:conv}
\subsection{Notation and set up}\label{sub:setup}
Fix \(f:\R^n\to\R^n\),   a polynomial vector field. We extend $\deg$ to polynomials as expected and set $\deg(f):=\max_{1\leq k\leq n} \deg(f_k)$.  We will assume from now onward  that
\begin{equation}\label{eq:hpf}
p:=\deg(f)\ge2\qquad \text{and}\qquad f(0)=0.
\end{equation}
Let  \(F_1=Df(0)\) be the jacobian of $f(x)$ evaluated at $x=0$, and decompose $f$ into its linear and nonlinear parts
\[
  f(x)=F_1x+f_{\ge2}(x) \quad\text{ where } \quad (f_{\ge2})_i(x)=\sum_{2\le |\gamma|\le p} c_{i\gamma}x^\gamma.
\]
Set
\[
  h:=p-1.
\]
So that the nonlinear part can increase the total degree of a monomial by at
most \(h\). As a measure of nonlinearity of the system, we consider the coefficient
\[
  C_2:=\max_{1\le i\le n}\sum_{2\le |\gamma|\le p}|c_{i\gamma}|.
\]
The  (\(\ell^\infty\)-induced) \emph{logarithmic norm}   of a $k\times k$ matrix $U$ is
\[
\mu_\infty(U)
  :=\max_{1\le i\le k}\left((U)_{ii}+\sum_{\ell\ne i}|(U)_{i\ell}|\right)\,.
\]
We let
\[
\mu_1:=
  \mu_\infty(F_1)
\]
be the logarithmic norm of the linear part (jacobian) of the system.
Fix a truncation degree \(D\ge1\). Let
\[
  \alpha=(\alpha_1,\ldots,\alpha_M)^T
\]
be the \emph{basis} vector of all  nonzero monomials of total degree at most \(D\),
ordered by nondecreasing total degree. The  Lie derivative of $\alpha$ taken componentwise
can be written as
\[
  \Lie_f(\alpha)=A\alpha +B\psi ,
\]
where \(\psi\) collects the monomials of degrees \(D+1,\ldots,D+p-1\) that are
not retained in the truncation. We decompose $A$ into a part \(A_1\)   induced by the linear vector field \(F_1x\), and a part \(A_2\)
induced by \(f_{\ge2}\). Formally, using an Iverson-bracket style notation, for $1\leq i,j\leq M$:
\[
  A=A_1+A_2, \quad\quad\text{ where: }\quad(A_1)_{ij}:=a_{ij}\cdot [\deg(\alpha_j)\leq 1]\quad\quad\quad  A_2:=A-A_1.
\]
%
For a row or column vector $q$, we let its \emph{support} be $\supp(q):=\{j:  q_j\neq 0\}$. For $J\subseteq \N$, we say that $q$ \emph{is supported on    degrees in} $J$ if $\{\deg(\alpha_j):j\in \supp(q)\}\subseteq J$.
Recalling  that
$g(x)=v^T\alpha(x)$,
we will assume from now onward that  the vector \(v^T\) is \emph{supported  on   degree 1}. For
example, for the coordinate observable \(g(x)=x_i\), one has \(\|v\|_1=1\). This is done only for notational simplicity.
Also recall the definition of error $\epsilon_D$ at truncation order $D$, for every $t$ in the interval of definition of $x(t;x_0)$:
\[
  z_D(t;z_0):=e^{At}\alpha(x_0),
  \qquad  \qquad
     \epsilon_D(t;x_0):=|v^T\alpha(x(t;x_0))-v^T z_D(t;z_0)|.
\]
Given a time horizon $T>0$ such that  $[0,T]$ is contained in the interval of definition of $x(t;x_0)$, we will assume an \emph{a priori} bound $R$ on the exact trajectory  on \([0,T]\)
\begin{equation}\label{eq:Rsup}
  \sup_{t\in[0,T]}\|x(t;x_0)\|_\infty\le R\,<+\infty,
  \qquad
  R_+:=\max(R,1).
\end{equation}
The \emph{support gap} for a degree-one observable is defined as:
\[
  k_D  
      :=\left\lceil\frac{D-h}{h}\right\rceil
\]
Since \(D\ge1\), this is equivalently
\[
  k_D+1=\left\lceil\frac{D}{h}\right\rceil.
\]
\ifmai
We will use the following  time scale adapted to the sign of $\mu_1$
\[
  \sigma:=
  \begin{cases}
    1, & \mu_1<0,\\
    h, & \mu_1\ge0.
  \end{cases}
\]
\fi
Define $\mu^{\text{+}}_1:=\max(\mu_1,0)$ and  $\mu^{{-}}_1:=\min(\mu_1,0)$, the positive and negative parts of $\mu_1$. For $t\in[0,T]$, define
\[
  \Theta(t):=\int_0^t e^{(\mu^{{-}}_1+h\mu^{\text{+}}_1) s}\,ds = \begin{cases}
\displaystyle \frac{1-e^{\mu_1t}}{-\mu_1}, & \mu_1<0,\\[1.2em]
t, & \mu_1=0,\\[0.8em]
\displaystyle \frac{e^{h\mu_1t}-1}{h\mu_1}, & \mu_1>0.
\end{cases}
\]
\ifmai
Equivalently,
\[
\Theta(t)=
\begin{cases}
\displaystyle \frac{1-e^{\mu_1t}}{-\mu_1}, & \mu_1<0,\\[1.2em]
t, & \mu_1=0,\\[0.8em]
\displaystyle \frac{e^{h\mu_1t}-1}{h\mu_1}, & \mu_1>0.
\end{cases}
\]
\fi
Finally, set
\[
  P_j:=\prod_{r=0}^{j-1}(1+rh),\qquad P_0:=1.
\]

\subsection{Statement of the results}\label{sub:stat}
We state the main results, a general bound (Theorem \ref{thm:finite}) and its corollaries giving   simpler, although slightly less tight, geometric convergence, at two different time regimes. We assume the notation fixed in the previous subsection.

\begin{theorem}[General bound]
\label{thm:finite}
Assume \eqref{eq:hpf} and \eqref{eq:Rsup} hold. For every \(D\ge1\) and every \(t\in[0,T]\),
\begin{equation}\label{eq:finite-refined-bound}
  \epsilon_D(t;x_0)
  \le
  \|v\|_1DC_2R^{D+1}R_+^{h-1}
  \sum_{j=k_D}^{D-1}P_jC_2^j
  \frac{\Theta(t)^{j+1}}{(j+1)!}.
\end{equation}
\ifmai
Consequently,
\begin{equation}\label{eq:finite-refined-bound}
\boxed{
  \sup_{t\in[0,T]}\epsilon_D(t;x_0)
  \le
  \|v\|_1DC_2R^{D+1}R_+^{h-1}
  \sum_{j=k_D}^{D-1}P_jC_2^j
  \frac{\Theta(T)^{j+1}}{(j+1)!}.
}
\end{equation}
\fi
\end{theorem}

In order to state the geometric corollaries, it is convenient to introduce
\[
  \rho(T):=hC_2\Theta(T).
\]

\begin{corollary}[First-tail geometric regime]
\label{cor:first-tail}
Assume \eqref{eq:hpf} and \eqref{eq:Rsup} hold. Assume
\[
  \rho(T)<1.
\]
Then, for all \(D\ge1\),
\begin{equation}\label{eq:first-tail-bound}
  \sup_{t\in[0,T]}\epsilon_D(t;x_0)
  \le K_1\,\gamma_1^{\,k_D},
\end{equation}
where
\[
  \gamma_1:=R^h \rho(T),
\quad
\text{and}\quad
  K_1:=\frac{\|v\|_1\rho(T)R^2R_+^{2h-2}}{1-\rho(T)}.
\]
In particular, if \(\gamma_1<1\), then
  $\sup_{t\in[0,T]}\epsilon_D(t;x_0)\to 0$
   as $D\to+\infty$.
\end{corollary}

\begin{corollary}[Last-tail geometric regime]
\label{cor:last-tail}
Assume \eqref{eq:hpf} and \eqref{eq:Rsup} hold. Assume
\[
  \rho(T)>1
  \qquad\text{and}\qquad
  R\rho(T)<1.
\]
Then, for all \(D\ge1\),
\begin{equation}\label{eq:last-tail-bound}
  \sup_{t\in[0,T]}\epsilon_D(t;x_0)
  \le K_2\,\gamma_2^{\,k_D},
\end{equation}
where
\[
  \gamma_2:=(R\rho(T))^h
\quad
\text{and}\quad
  K_2:=\frac{\|v\|_1\rho(T)RR_+^{h-1}}{\rho(T)-1}.
\]
Consequently,
  $\sup_{t\in[0,T]}\epsilon_D(t;x_0)\to 0$
   as $D\to+\infty$.
\end{corollary}

\begin{remark}[On trajectory bounds, the transition case, and infinite horizons]{\em
\label{rem:R-critical-infinite}
We record three comments on the hypotheses and on the geometric corollaries.

\begin{enumerate}
\item \emph{On the a priori bound \(R\).}
The assumption
\[
  \sup_{t\in[0,T]}\|x(t;x_0)\|_\infty\le R
\]
should not be viewed as a serious limitation of the method. Such a bound can be
obtained analytically, for instance by Lyapunov estimates, comparison arguments
or Gronwall-type inequalities, and it can also be obtained by certified numerical
methods, such as interval integration, Taylor-model enclosures, or reachability
tools producing validated flowpipe enclosures
\cite{Chicone,Nediakov,TaylorModels,TaylorFlow,Flow,CORA,ReachReview}.
In a reachability computation, \(R\) would typically be chosen locally from the
current reachset and the current time step. Thus, for an autonomous polynomial
system, the truncated Carleman matrix is fixed once the truncation degree is
chosen, while the trajectory bound and the corresponding error enclosure may be
updated along the flowpipe.

This type of assumption is also standard in related Carleman error estimates.
For instance, the backward-integration estimate of Forets--Pouly uses an
a priori bound on the solution, and the local-bound version of the estimates of
Amini--Sun--Motee and collaborators similarly assumes a bound on the trajectory
over the time interval under consideration
\cite{Forets2,Amini,AminiZhengSunMotee2025}.

\item \emph{On the transition value \(\rho=1\).}
The two geometric corollaries are stated separately in the regimes
\[
  \rho:=\rho(T)<1
  \qquad\text{and}\qquad
  \rho>1 .
\]
The apparent critical behaviour at \(\rho=1\) is an artifact of replacing the
finite geometric sum by one-sided tail estimates. Indeed, the proof actually
produces a finite sum of the form
\[
  \sum_{j=k_D}^{D-1}\rho^j,
\]
which is perfectly well defined at \(\rho=1\), where it equals \(D-k_D\). One
could therefore state a single estimate using the exact finite geometric sum
\[
  \sum_{j=k_D}^{D-1}\rho^j
  =
  \begin{cases}
  \dfrac{\rho^{k_D}-\rho^D}{1-\rho}, & \rho\ne1,\\[1.1em]
  D-k_D, & \rho=1.
  \end{cases}
\]
We have kept the simpler first-tail and last-tail statements because they give
clean geometric factors and are easier to read. The transition case can be
handled directly from the finite-sum estimate, at the cost of slightly heavier
notation.

\item \emph{On infinite time horizons.}
The present theorem can also yield convergence on \([0,+\infty)\), but only in a
genuinely contractive regime. Assume that
\[
  \mu_1<0
  \qquad\text{and}\qquad
  R_\infty:=\sup_{t\ge0}\|x(t;x_0)\|_\infty<\infty .
\]
Then
\[
  \Theta_\infty:=\lim_{T\to\infty}\Theta(T)=\frac1{-\mu_1},
  \qquad
  \rho_\infty:=(p-1)C_2\Theta_\infty
  =
  \frac{(p-1)C_2}{-\mu_1}.
\]
A simple sufficient condition for convergence uniformly on the infinite time
horizon is
\[
  R_\infty
  \max\left\{
    \rho_\infty,\,
    \rho_\infty^{1/(p-1)}
  \right\}
  <1 .
\]
Equivalently, in the first-tail regime \(\rho_\infty<1\), it is enough that
\[
  R_\infty^{p-1}\rho_\infty<1,
\]
whereas in the last-tail regime \(\rho_\infty>1\), it is enough that
\[
  R_\infty\rho_\infty<1.
\]
At the borderline \(\rho_\infty=1\), the exact finite-sum estimate still gives
convergence provided \(R_\infty<1\). For \(\mu_1\ge0\), the time factor
\(\Theta(T)\) is unbounded as \(T\to+\infty\), and the present estimates do not
give an infinite-horizon convergence guarantee.
\end{enumerate}}
\end{remark}

\subsection{Proof overview}
We summarize below the structure of the argument for the main theorem and its corollaries, referring to
Section~\ref{sec:proofs} for the detailed derivation. The proof proceeds in four main
steps.

\begin{enumerate}
\item \textbf{Variation of parameters.}
The exact Carleman coordinates  $\alpha(x(t;x_0))$ satisfy a
nonhomogeneous linear system \eqref{eq:system2}.
The truncated Carleman approximation $z_D(t;x_0)$ solves the corresponding homogeneous
system \eqref{eq:carltrunc}.   Hence the error vector $\alpha(x(t;x_0))- z_D(t;x_0)$ solves the difference between these two systems. By variation of parameters \cite[Prop.2.67]{Chicone}
\[
\epsilon_D(t;x_0)=\big|v^T(\alpha(x(t;x_0))-z_D(t;x_0))\big|
=
|\int_0^t v^Te^{A(t-\tau)}w(\tau)\,d\tau|\leq  \int_0^t |v^Te^{A(t-\tau)}w(\tau)|\,d\tau
\]
where \(w(\tau)=B\psi(x(\tau;x_0))\) collects the overflow terms generated by monomials of degree
\(>D\).  Our main task is estimating the integrand $v^Te^{A(t-\tau)}w(\tau)\,d\tau$.

\item \textbf{Dyson expansion.}
Rather than Taylor-expanding the matrix exponential $e^{A(t-\tau)}$, we consider a type of \emph{Dyson expansion} \cite{Dyson,Pazy1983}, which in its expression keeps   the effect of the linear ($A_1$) and nonlinear ($A_2$) parts distinct.   We will have to wait for the next step of the proof (step 3) to see the advantage of doing so. For now, after splitting    the Carleman matrix as
$
A=A_1+A_2,
$
we expand \(e^{(A_1+A_2)s}\) ($s:=t-\tau$)   by the Dyson-Duhamel formula, which  allows us to rewrite the integrand in the variation of parameters expression above  as
{\small
\[
v^T e^{(A_1+A_2)s}w(\tau)
=
\sum_{j=0}^{\infty}
\int_{\Delta_j(s)}
v^T e^{A_1(s-s_j)}
A_2e^{A_1(s_j-s_{j-1})}
\cdots
A_2e^{A_1s_1}\,w(\tau)
\,ds_1\cdots ds_j
\]
}\noindent
where
$
\Delta_j(s)=\{0\le s_1\le\cdots\le s_j\le s\}\subseteq \reals^j
$.
Consider  the multiplications  inside the integral  above, from left to right. Starting from the vector $v^T$ supported on degree 1, each $e^{A_1\theta}$ leaves the support degree  unchanged, while each $A_2$ causes an increase   between 1 and $h$. So the final row vector, before $w(\tau)$, is supported only on degrees between $1+j$ and $\min\{1+hj,D\}$.
On the other hand, the overflow term \(w(\tau)\) is supported only on degrees at least \(D-h+1\).
Therefore all   summands with  index $j$ too small or too large vanish: explicitly, we are left with indices $j$ in the range
\[
\left\lceil\frac{D-h}{h}\right\rceil=k_D\leq j\leq D-1.
\]

\item \textbf{Estimates of the non-vanishing Dyson terms.}
We now have to estimate  the remaining summands  with indices $j$ in the above specified range. We rely on matrix bound  estimates. 
Starting from $v^T$ and considering the  factors from left to right, we track   the support degree   reached by the resulting row vector  at each point. Before the \(r\)-th
occurrence of \(A_2\), the   support of the current row vector say $q$ has reached degree at most \(1+rh\), so each such occurrence contributes a factor bounded by
\[
\|q A_2\|_1
\le
\|q \|_1 (1+rh)C_2.
\]
The factors contributed by the occurrences of the exponentials \(e^{A_1\theta}\)    are also estimated degree by
degree, using the logarithmic norm of the linear part:
\[\|q e^{A_1\theta}\|_1\le \|q\|_1e^{b\mu_1\theta}\]
where $b$ bounds the support degree of $q$, from above if $\mu_1\geq 0$ and from below if $\mu_1<0$. Crucially, if $\mu_1<0$ \emph{these factors are   exponentially decreasing with} $\theta$: this is where the splitting $A_1+A_2$ shows its benefit. Finally, $w(\tau)=B\psi(x(\tau;x_0))$ is bounded by $\approx C^D$, where the constant $C$ depends on the norm of $B$   and on the bound $R$ on $\|x(\tau;x_0)\|_\infty$. After factoring out  the bounds $(1+rh)C_2$ and $C^D$, not depending on some $s_i$,   the integrand  in the $j$-th term of the Dyson expansion reduces to a product of exponentials, which is proven to reduce to
$
\Theta(t)
$.
\ifmai
\begin{cases}
\dfrac{1-e^{\mu_1T}}{-\mu_1}, & \mu_1<0,\\[1.2em]
T, & \mu_1=0,\\[0.8em]
\dfrac{e^{(p-1)\mu_1T}-1}{(p-1)\mu_1}, & \mu_1>0.
\end{cases}
\]
\fi

\item \textbf{Summation and extraction of geometric rates.}
Inserting the refined Dyson estimate into the variation-of-parameters formula gives
the finite bound stated in \eqref{eq:finite-refined-bound}. The resulting sum
contains the product
\[
P_j=\prod_{r=0}^{j-1}(1+rh).
\]
Using
\[
P_j\le h^j j!,
\]
one obtains the geometric bounds stated in
\eqref{eq:first-tail-bound} and \eqref{eq:last-tail-bound}. In the first-tail
regime,
$
\rho:=\rho(T)<1,
$
the corresponding geometric factor is
\[
\gamma_1
=
R^{h}\rho.
\]
Thus, if \(\gamma_1 <1\), the truncation error converges to zero exponentially in
\(k_D\), hence exponentially in \(D\). In the complementary last-tail regime, the
geometric factor is the one displayed in \eqref{eq:last-tail-bound}; this covers
cases where \(\rho>1\) but the product \(R\rho\) is still smaller than one.
\end{enumerate}

\section{Experiments and comparison}\label{sec:exp}
We consider two standard systems with qualitatively
different behaviour: the Stuart--Landau system, where the linear part is
contractive in logarithmic norm, and the Van der Pol system, whose dynamics
exhibit a limit cycle.
Our error bounds on these systems should be compared with the explicit estimates of Forets and
Pouly~\cite{Forets2}.  Amini,
Sun and Motee~\cite{Amini} and Amini, Zheng, Sun and Motee
\cite{AminiZhengSunMotee2025} also provide bounds applicable to general, possibly non-polynomial systems, under coefficient-decay
assumptions. However, numerical experiments show that their bounds are significantly more conservative than those of~\cite{Forets2} when specialized to polynomial systems.

Compared with Theorems 4.2 and 4.3 of Forets--Pouly~\cite{Forets2}, our
analysis keeps the original polynomial degree structure explicit. Theorem 4.2
uses backward integration, strongly reminiscent of   Bellman's method \cite{Bell}.
This result can exploit
logarithmic-norm contractivity, but it is formulated after quadratic lifting. Like ours, it
depends on an a priori trajectory bound. Theorem 4.3 removes this trajectory
bound by means of a generating-function argument, but at the price of more
global coefficient estimates. Our proof instead works directly in the original
monomial basis and for a prescribed observable. By decomposing the Carleman
matrix into its degree-preserving linear part and degree-raising nonlinear part,
and by expanding the exponential with a Dyson--Duhamel formula, we obtain
degree-local estimates that track how the discarded tail can propagate back to
the observable. Thus the main gain is structural: the bound is
observable-specific, avoids a preliminary quadratic reduction, and retains the
logarithmic-norm information of the original linear dynamics.

\subsection{Example 1: Van der Pol system}\label{vdp}
\noindent
We consider the Van der Pol (VDP) system in the polynomial form
\begin{equation}\label{eq:VDP}
	\begin{aligned}
		\dot x_1 &= -\frac{x_1^3}{3}+x_1-x_2,\\
		\dot x_2 &= x_1,
	\end{aligned}
\end{equation}
with initial condition $x_0=(0.5,0)^T$. This is the same system considered in \cite{Forets2}. As an observable, set
$g(x)=x_1$, thus  \(\|v\|_1=1\).
Since \eqref{eq:VDP} is a polynomial vector field of degree \(p=3\), we have \(h=p-1=2\) and
\begin{equation}
k_D
=
\max\!\left\{
\left\lceil
\frac{D-p+1}{p-1}
\right\rceil,
0
\right\}
=
\max\!\left\{
\left\lceil
\frac{D-2}{2}
\right\rceil,
0
\right\}.
\end{equation}
The linear component of \eqref{eq:VDP} is characterized by the Jacobian matrix evaluated at the origin, that is given by
\[F_1=
\begin{pmatrix}
	1 & -1\\
	1 & 0
\end{pmatrix},
\]
with logarithmic norm
$\mu_1=\mu_\infty(F_1)=\max\{1+|-1|,\;0+|1|\}=2>0$.
On the other hand, the nonlinear component of \eqref{eq:VDP}  is
$f_{\ge 2}(x)=
(
	-\dfrac{x_1^3}{3},\,
	0)^T$,
giving
$C_2
=
\max_i
\sum_{2\le |\gamma|\le 3}
|c_{i\gamma}|
=
\frac13.
$
Since \(\mu_1>0\), we get
\[\Theta(T)
=
\frac{e^{h\mu_1T}-1}{h\mu_1}
=
\frac{e^{4T}-1}{4}.\]

\paragraph{Maximal theoretical convergence horizon}
In our framework, the convergence guarantees for the error bounds depend on the quantities
$h = p-1$,
$
\rho(T)  = hC_2\Theta(T)$ and
\[
R(T) := \sup_{0\le t\le T}\|x(t;x_0)\|_\infty.
\]
(we implicitly require $x(t;x_0)$ well defined for $0\leq t\le T$).
The geometric  Corollaries \ref{cor:first-tail} and \ref{cor:last-tail} establish convergence under two distinct regimes:
\begin{itemize}
	\item \emph{First-tail regime}: if $\rho(T)<1$ and
	$\gamma_1(T):=R(T)^h\rho(T)<1$,
	\item \emph{Last-tail regime}: if $\rho(T)>1$ and
	$R(T)\rho(T)<1$.
\end{itemize}
Equivalently, the second regime can be expressed as
$\gamma_2(T):=(R(T)\rho(T))^h < 1$.
The maximal achievable convergence horizon guaranteed by these results is therefore (convening that $\sup \es:=0$)
{\small
\begin{equation}\label{eq:Tmax}
	T_{\max} = \max\!\Bigl(
	\sup\{T\ge0: \rho(T)<1,\; R(T)^h\rho(T)<1\},\;
	\sup\{T\ge0: \rho(T)>1,\; R(T)\rho(T)<1\}
	\Bigr).
\end{equation}
}\noindent
The value $T_{\max}$ can be computed numerically.  For the Van der Pol system,   the maximum is attained above the critical point $T^*=\frac 1 4 \log 7\approx 0.4864$ where $\rho(T^*)=1$,  in the last-tail regime:
\[T_{\max} \approx 0.561.\]

\paragraph{Numerical illustration and comparison with Forets--Pouly}
Consider the Carleman linearization \eqref{eq:carltrunc}, with  a `budget' fixed to  $M=150$  monomials in $\alpha$.  Let
\[d(M,n) := \max \left\{ d \geq 1 : \binom{n+d}{d} - 1 \leq M \right\}\]
that is,   the maximum degree $d$  such that there are  at most  $M$  non-zero monomials  in $n$ variables  of degree $\leq d$. In terms of $D$,  $M=150$ allows us a truncation degree of $D=d(150,2)=16$.
We  report numerical error bounds for the VDP system, across three   time horizons:  $T=0.2$, $T=0.4$, and $T=0.5$, all within the theoretical maximum $T_{\max}\approx 0.561$. This allows us to observe the behavior of the theoretical bounds in both the first- and last-tail regimes.
A certified upper bound $R$  on $\|x(t;x_0)\|_\infty$ for $t\in [0,T]$ can be computed in each case   by   using either a validated numerical enclosure \cite{Nediakov} or an analytic comparison estimate. A comparison-based approach yields the conservative bounds $R=0.505$ on $[0,0.2]$, $R=0.72$ on $[0,0.4]$ and  $R=0.78$ on $[0,0.5]$; we will stick to these values below. %
In each case, we compute both the  bound on the error $|\epsilon_D|$ given by Theorem \ref{thm:finite} as well as those given by the geometric corollaries. 
Note that $T=0.2, 0.4$ are below the critical time $T^*\approx  0.486$ and fall therefore in the first-tail regime of Corollary \ref{cor:first-tail}, while $T=0.5$ falls in the last-tail regime of Corollary \ref{cor:last-tail}.

We also
 compare our bounds with  those provided by Forets and Pouly~\cite[Th.4.2,Th.4.3]{Forets2}, with the same monomial budget   $M=150$, across the three values of $T$ considered above.
 %
 The method of \cite{Forets2} requires lifting the system to a quadratic form, by introducing auxiliary variables corresponding to   products of the original variables. This leads to an expansion from the original $n=2$ to $\tilde n=5$ of the number of state variables\footnote{Granting \cite{Forets2} that, in an actual implementation of their method,  one  can identify Kronecker products up to commutativity of variables, i.e. $x_1x_2=x_2x_1$.}. Given a budget of $M$, one should   compare our bounds and that of \cite{Forets2} at truncation orders of $D=d(M,n)$ and $N=d(M,\tilde n)$, respectively. In the present case, $\tilde n=5$, hence $N=d(150,  5)=5$.
Regarding the maximal theoretical convergence time horizon,~\cite[Th.4.3]{Forets2} provides a static bound given by the following. The matrices  $\tilde F_1$, $\tilde F_2$ refer to the quadratic system (see \cite[Sect.4.4]{Forets2} for an  explicit   definitions).
 \begin{equation}
 	\label{TFP43}
 	T_{\max}^{\text{FP4.3}}  = \frac{1}{\|\tilde F_1\|_\infty} \log \left( 1 + \frac{\|\tilde F_1\|_\infty}{\|x_0\|_\infty  \|\tilde F_2\|_\infty} \right)=0.641.
 \end{equation}
In contrast, the maximal theoretical convergence time horizon of~\cite[Th.4.2]{Forets2},  is given by  the following, where $\gamma_{\mathrm{FP}}(t):=R(t)\|\tilde F_2\|_\infty (e^{\mu_\infty(\tilde F_1)t}-1)/\mu_\infty(\tilde F_1)$ is their geometric basis (cf. their formula\footnote{They denote $R(t)$ as  $\|x\|_t$. The condition on  $\gamma_{\mathrm{FP}}$ is implicit in their formulation.} (25)):
 \begin{equation}\label{TFP42}
 	T_{\max}^{\text{FP4.2}} = \sup \left\{ t \ge 0 : \phi(t):= t - \frac{1}{\mu_\infty(\tilde F_1)} \log \left( 1 + \frac{\mu_\infty(\tilde F_1)}{\|\tilde F_2\|_\infty R(t)} \right) < 0 \text{ and }\gamma_{\mathrm{FP}}(t)<1  \right\}.
 \end{equation}
Solving numerically, we find:
\[
T_{\max}^{\text{FP4.2}} \approx 0.561.
\]
The numerical results and the comparison are summarized in   Table \ref{tab:vdp}.

\begin{table}[t]{
	\centering	
	\label{tab:vdp}
	\vspace{2mm}
	\setlength{\tabcolsep}{5.0pt}
	\begin{tabular}{ccccc cccc}
		\toprule
		\multicolumn{3}{c}{\textbf{$T_{\max}$}}&	\multirow{2}{*}{\textbf{T}} & \multirow{2}{*}{\textbf{R}} 	 &\multicolumn{4}{c}{\textbf{$\epsilon_D$}}\\
		\cmidrule(lr){1-3} \cmidrule(lr){6-9} 
		\textbf{Cor1,2} &\textbf{FP 4.2} & \textbf{FP 4.3} &&& \textbf{Th2} &\textbf{Cor1,2} &  \textbf{FP 4.2} & \textbf{FP 4.3}\\ 
		\midrule
		
		\multirow{3}{*}{$0.562$}& \multirow{3}{*}{$0.561$} & \multirow{3}{*}{$\mathbf{0.641}$} &$0.2$&{$0.505$} & $\mathbf{6.92 \cdot 10^{-12}}$&${6.81 \cdot 10^{-11}}$ & $5.90  \cdot 10^{-6}$ &$1.38\cdot 10^{-5}$\\
		[1ex]
		
		&&&$0.4$ &$0.720$& $ \mathbf{6.31\cdot 10^{-5}}$ & $5.43\cdot 10^{-4}$ &$1.73 \cdot 10^{-2}$ & $1.43 \cdot 10^{-2}$\\
		[1ex]
		
		&&&$0.5$&$0.780$& $\mathbf{3.32 \cdot 10^{-2}} $  &  $9.52 \cdot 10^{-1} $  &$3.08 \cdot 10^{-1}$  &$3.38 \cdot 10^{-1}$\\
		
		\bottomrule
	\end{tabular}
	\caption{Van Der Pol system: maximal theoretical convergence horizon ($T_{\max}$), truncation errors bound ($\epsilon_D$), and geometric bases provided by our results (Th. \ref{thm:finite}, Cor. \ref{cor:first-tail},\ref{cor:last-tail})  and by Forets--Pouly's (FP) (Th. 4.2,4.3), across different time horizons $T$. We fix $M=150$, $D=d(M,2)$, $N=d(M,5)$. Best results in boldface.}}
\end{table}



To further analyze the   behavior of our error bound against FP's as the budget $M$  increases, we plot in Figure~\ref{fig:vdp_plot}  the error ratio between our direct bound (Theorem 1) and the FP bounds, as a function of the   budget $M$.   The plots show that the error ratio remains well below $1$ for all the analyzed time horizons and all values of $M$.  Furthermore, the error ratios decrease  as $M$ grows, confirming that our approach  scales more effectively  with $M$.

\begin{figure}[htbp]
	\centering
	\includegraphics[width=0.32\textwidth]{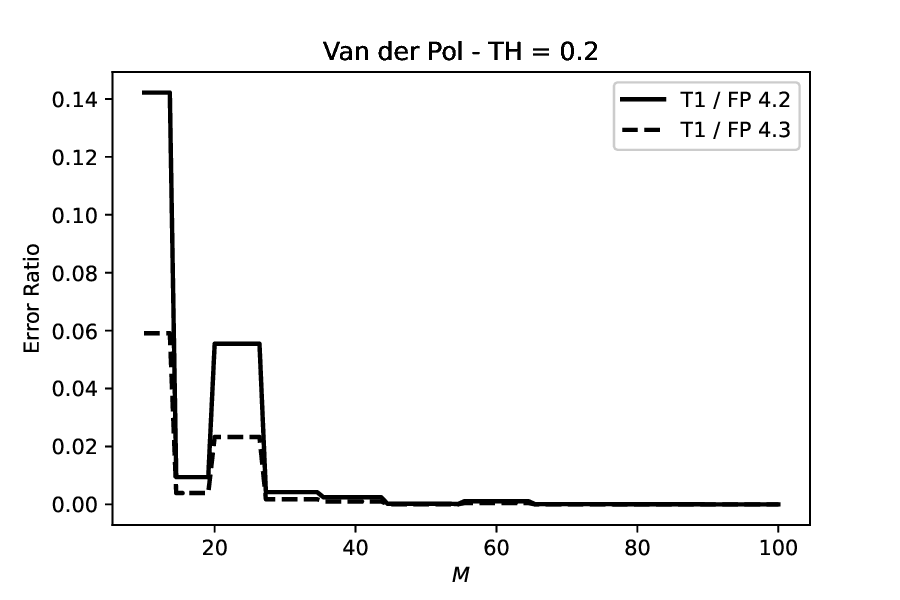}
	\includegraphics[width=0.32\textwidth]{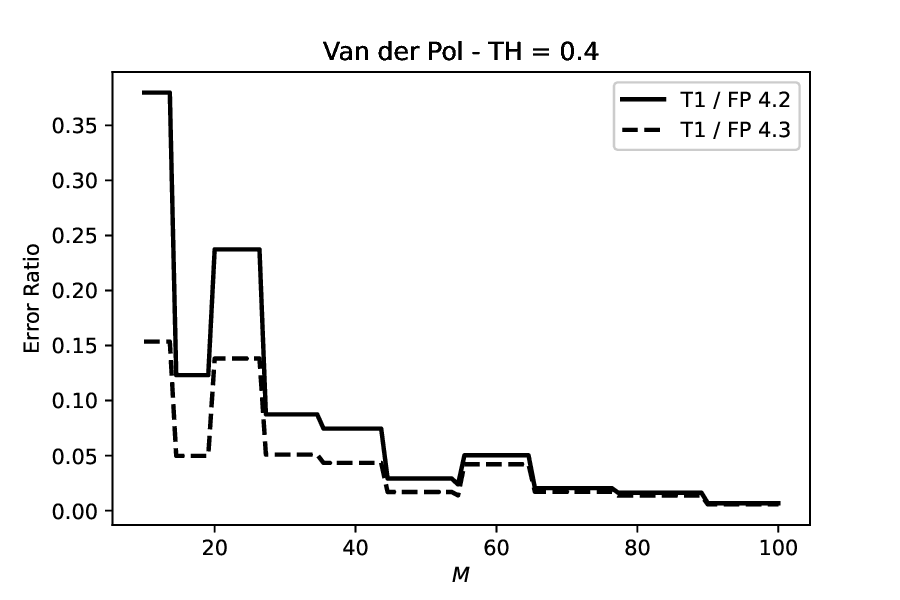}
	\includegraphics[width=0.32\textwidth]{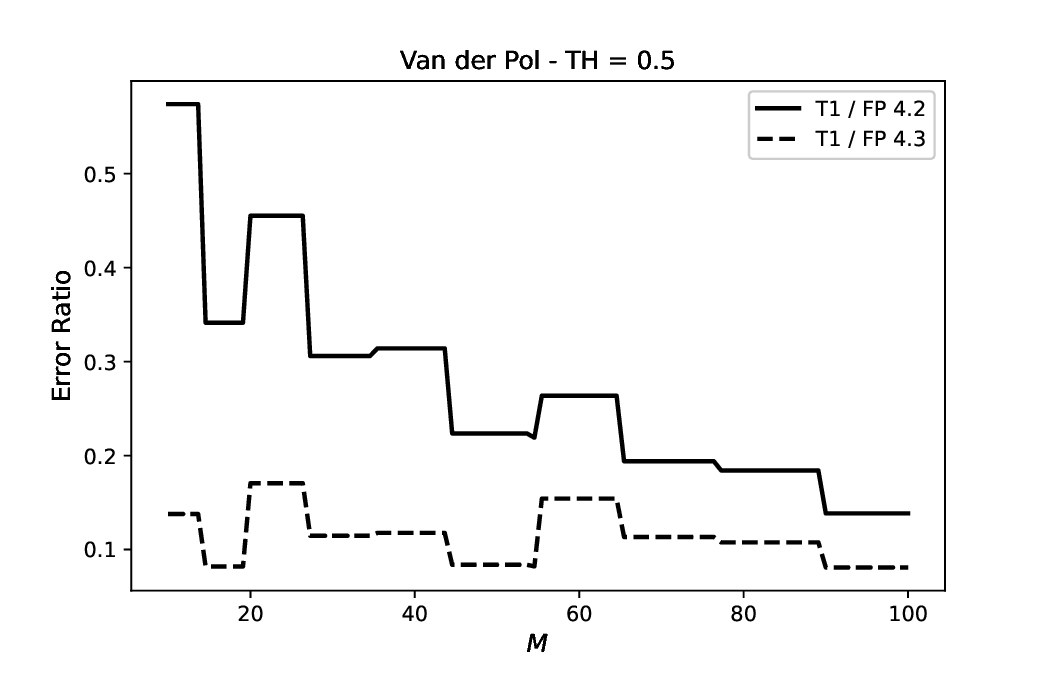}
	\caption{Van der Pol system: error ratios   (bound in Theorem~\ref{thm:finite})$/$(bound in~\cite[Th.x]{Forets2}),  for x=4.2  (solid line)  and   x=4.3 (dashed line), as a function of the   budget $M$, evaluated at $T=0.2$ (left), $T=0.4$ (center) and $T=0.5$ (right).}
	\label{fig:error_ratio_VDP}
	\label{fig:vdp_plot}
\end{figure}

\subsection{Example 2: Stuart-Landau system}
We now consider the Stuart-Landau  (SL) system
\begin{equation}\label{eq:SL}
	\begin{aligned}
		\dot x_1 &= -2x_1-x_2-x_1^3-x_1x_2^2,\\
		\dot x_2 &= x_1-2x_2-x_1^2x_2-x_2^3,
	\end{aligned}
\end{equation}
with initial condition $x_0=(0.5,0)^T$. We take as observable
$g(x)=x_1$, thus 
\(\|v\|_1=1\).
System \eqref{eq:SL} is cubic, then $p=3$, $h:=p-1=2$ and
\begin{equation}
k_D
=\max \left\{ \left\lceil\frac{D-p+1}{p-1}\right\rceil,0 \right\}
= \max \left\{\left\lceil\frac{D-2}{2}\right\rceil,0\right\}.
\end{equation}
The linear dynamics of~\eqref{eq:SL} are characterized by
\begin{equation}
F_1=
\begin{pmatrix}
	-2 & -1\\
	1 & -2
\end{pmatrix},\end{equation}
with
$\mu_1
=
\mu_\infty(F_1)
=
\max\{-2+|-1|,\,-2+|1|\}
=
-1<0$. This can be regarded, informally speaking, as a weakly contractive system, that is, contractive in its linear part.  Instead, the nonlinear part is given by
$f_{\ge2}(x)
=
	(-x_1^3-x_1x_2^2,\,
	-x_1^2x_2-x_2^3)^T$,
giving $C_2=2$.
Since, \(\mu_1<0\), the time factor is \[\Theta(T)
=
\frac{1-e^{\mu_1T}}{-\mu_1}
=
1-e^{-T}.\]

\paragraph{Maximal theoretical convergence horizon}
Following the same methodology applied to the VDP example, we find the critical time $T^*=\log(4/3)\approx 0.2876$, while formula \eqref{eq:Tmax} gives the maximal theoretical convergence   horizon
\[
T_{\max}\approx 0.693.
\]

\paragraph{Numerical illustration and comparison with Forets--Pouly}
Analogously to the VDP case, we investigate the numerical error bounds for the SL system for the three representative time horizons: $T=0.2$, $T=0.4$ and $T=0.5$. We keep the monomial budget fixed at $M=150$.
The dissipative structure of system \eqref{eq:SL} guarantees that $
\|x(t;x_0)\|_\infty \le\frac12$ for all $t\ge 0$: so we just take  $R=\frac12$ as a uniform upper bound over the considered time horizons. The results of our analysis are reported in Table \ref{tab:LS}. In terms of convergence horizons, the standard FP formulation restricts the validity of Theorem~4.3 to a fixed and highly restrictive horizon of $T_{\max} \approx 0.231$.
For this reason, the error values for FP Theorem 4.3 are omitted for $T=0.4$ and $T=0.5$, as these horizons strictly exceed its  convergence limit. On the contrary, our method remains robust over all selected time intervals.  We plot in Figure~\ref{fig:plot_sl}  the error ratio between our direct bound (Theorem 1) and the FP bounds, as a function of the   budget $M$. The error ratios remain significantly below the threshold of $1$, staying well below $0.4$ in all the considered configurations. Therefore, our approach systematically provides tighter error bounds. Furthermore, as the dimension of the truncated system increases, the curves exhibit a monotonic decay toward zero.

\begin{table}[t]{
	\centering	
	\label{tab:LS}
	\vspace{2mm}
	\setlength{\tabcolsep}{5.0pt}
	\begin{tabular}{llccc cccc}
		\toprule
		\multicolumn{3}{c}{\textbf{$T_{\max}$}} &\multirow{2}{*}{\textbf{T}}&\multirow{2}{*}{\textbf{R}} & \multicolumn{4}{c}{\textbf{$\epsilon_D$}}\\
		\cmidrule(lr){1-3} \cmidrule(lr){6-9} 
		\textbf{Cor1,2} & \textbf{FP 4.2} & \textbf{FP 4.3} 	&&& \textbf{Th2} & \textbf{Cor1,2} & \textbf{FP 4.2} & \textbf{FP 4.3}\\
		\midrule
		\multirow{3}{*}{$\mathbf{0.693}$}& \multirow{3}{*}{$\mathbf{0.693}$} & \multirow{3}{*}{$0.231$} &$0.2$&\multirow{3}{*}{$0.500$} & $\mathbf{3.15 \cdot 10^{-7}}$&${4.24 \cdot 10^{-6}}$ & $3.13  \cdot 10^{-3}$ &$2.03$\\
		[1ex]
		
		&&  &$0.4$ &&$\mathbf{2.38  \cdot 10^{-4}}$&$6.07 \cdot 10^{-3}$ &$6.23 \cdot 10^{-2}$&$-$\\
		[1ex]
		
		& & &$0.5$ &&$\mathbf{2.58\cdot 10^{-3}}$&${4.79\cdot 10^{-2}}$  &$1.51  \cdot 10^{-1}$&$-$\\
		\bottomrule
	\end{tabular}
\caption{Stuart-Landau system: maximal theoretical convergence horizon ($T_{\max}$), truncation errors bound ($\epsilon_D$), and geometric bases provided by our results (Th. \ref{thm:finite}, Cor. \ref{cor:first-tail},\ref{cor:last-tail})  and by Forets--Pouly's (FP) (Th. 4.2,4.3), across different time horizons $T$. We fix $M=150$, $D=d(M,2)$, $N=d(M,5)$. Best results in boldface.}}
\end{table}

\begin{figure}[t]
	\centering
	\includegraphics[width=0.32\textwidth]{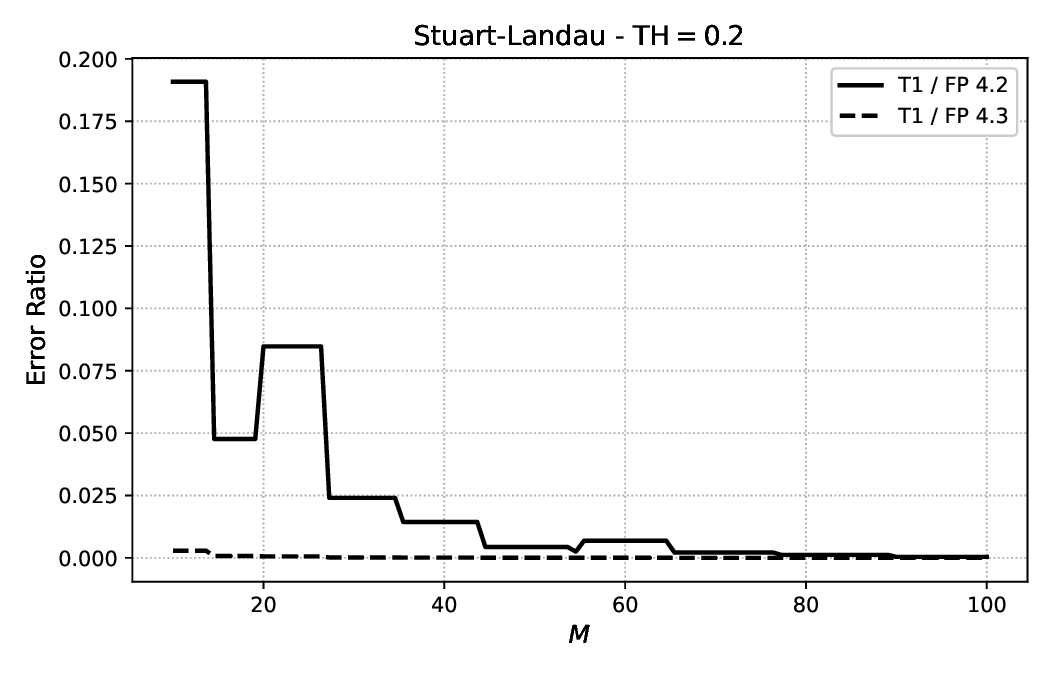}
	\includegraphics[width=0.32\textwidth]{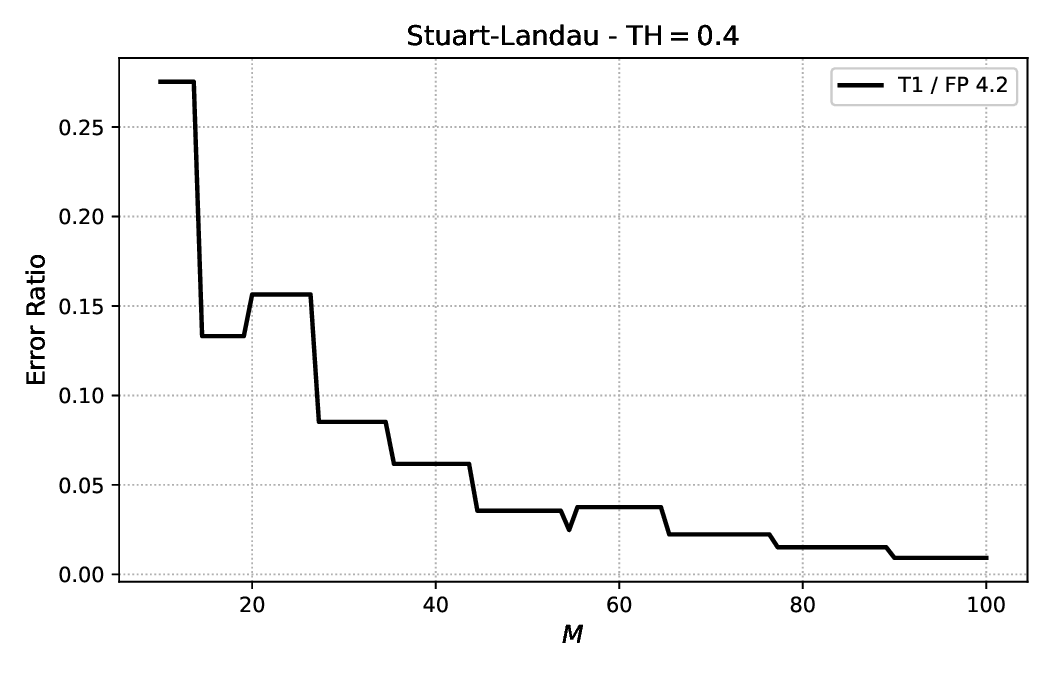}
	\includegraphics[width=0.32\textwidth]{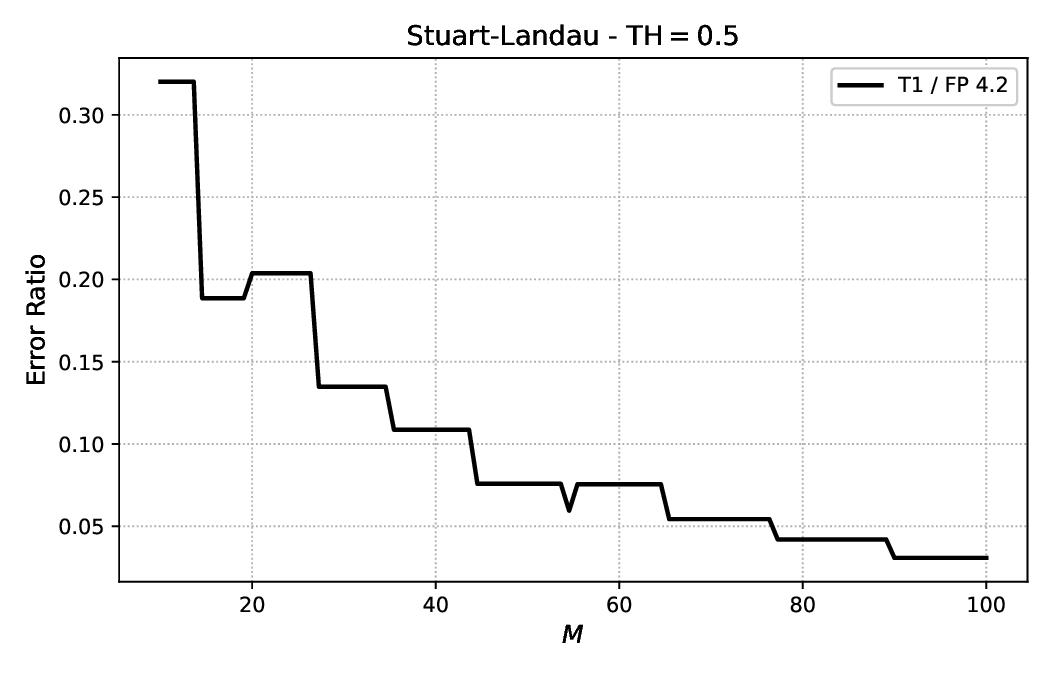}
	\caption{Stuart-Landau system:  error ratios   (bound in Theorem~\ref{thm:finite})$/$(bound in~\cite[Th.x]{Forets2}),  for x=4.2  (solid line)  and   x=4.3 (dashed line, only for $T=0.2$), as a function of the   budget $M$, evaluated at $T=0.2$ (left), $T=0.4$ (center) and $T=0.5$ (right).}
	\label{fig:plot_sl}
\end{figure}

\subsection{Discussion}


Overall, our error estimates are substantially less conservative than the
norm-based/power-series Carleman bounds of~\cite[Th.~4.3]{Forets2}, and also
than the coefficient-based bounds of
\cite{Amini,AminiZhengSunMotee2025}, which we do not report here for brevity.
The comparison with the trajectory-dependent backward estimate of
\cite[Th.~4.2]{Forets2} is more nuanced: in our experiments the admissible time
horizons are essentially comparable, because both approaches can exploit
contractivity when available. The main difference is structural. Our estimates
are derived directly in the original monomial basis, for the chosen observable,
and by tracking degree propagation explicitly. This yields geometric convergence
factors that are competitive with those of~\cite[Th.~4.2]{Forets2}, and in some
cases sharper, while avoiding a preliminary quadratic reduction.


\ifmai
\fi

\ifmai
\subsection{Example: the Stuart--Landau system}
We consider the Stuart--Landau system
\begin{equation}\label{eq:SL}
\begin{aligned}
\dot x_1 &= -2x_1-x_2-x_1^3-x_1x_2^2,\\
\dot x_2 &= x_1-2x_2-x_1^2x_2-x_2^3,
\end{aligned}
\end{equation}
with initial condition
\[
x_0=(1/2,0)^T.
\]
This is the same Stuart--Landau instance considered in the comparison section.

We take as observable
\[
g(x)=x_1.
\]
Hence \(d_g=1\) and \(\|v\|_1=1\).

The system is cubic, so
\[
p=3,\qquad h:=p-1=2.
\]

The linear part is
\[
F_1=
\begin{pmatrix}
-2 & -1\\
1 & -2
\end{pmatrix}.
\]
Therefore, in the \(\ell_\infty\)-logarithmic norm,
\[
\mu_1
=
\mu_\infty(F_1)
=
\max\{-2+|-1|,\,-2+|1|\}
=
-1.
\]
Thus this example lies in the contractive case \(\mu_1<0\).

The nonlinear part is
\[
f_{\ge2}(x)
=
\begin{pmatrix}
-x_1^3-x_1x_2^2\\
-x_1^2x_2-x_2^3
\end{pmatrix}.
\]
Each component has two cubic monomials with coefficient of absolute value \(1\). Hence
\[
C_2
=
\max_i \sum_{2\le |\gamma|\le 3}|c_{i\gamma}|
=
2.
\]

\paragraph{Trajectory bound.}
Let
\[
r^2=x_1^2+x_2^2.
\]
Then
\[
\frac{d}{dt}r^2
=
2x_1\dot x_1+2x_2\dot x_2.
\]
Substituting \eqref{eq:SL},
\[
\begin{aligned}
\frac{d}{dt}r^2
&=
2x_1(-2x_1-x_2-x_1^3-x_1x_2^2)
+
2x_2(x_1-2x_2-x_1^2x_2-x_2^3)\\
&=
-4x_1^2-2x_1x_2-2x_1^4-2x_1^2x_2^2
+
2x_1x_2-4x_2^2-2x_1^2x_2^2-2x_2^4\\
&=
-4(x_1^2+x_2^2)-2(x_1^2+x_2^2)^2\\
&=
-4r^2-2r^4
\le 0.
\end{aligned}
\]
Therefore
\[
r(t)\le r(0)=\frac12
\qquad\text{for all }t\ge0.
\]
In particular,
\[
\|x(t;x_0)\|_\infty\le \|x(t;x_0)\|_2\le\frac12.
\]
Thus we may take
\[
R=\frac12,
\qquad
R_+=1.
\]

For \(\mu_1=-1\), the refined time factor is
\[
\Theta(T)
=
\frac{1-e^{\mu_1T}}{-\mu_1}
=
1-e^{-T}.
\]
For convenience, write
\[
\theta:=\Theta(T)=1-e^{-T}.
\]

Since \(d_g=1\) and \(h=2\), the support-gap index is
\[
k_D
=
\left\lceil\frac{D-p+1}{p-1}\right\rceil\vee0
=
\left\lceil\frac{D-2}{2}\right\rceil\vee0.
\]
For \(D\ge2\), this is simply
\[
k_D=\left\lceil\frac{D-2}{2}\right\rceil.
\]

\paragraph{Explicit finite-\(D\) refined bound.}
The refined finite bound gives
\[
\epsilon_D(t)
\le
D C_2 R^{D+1}R_+^{p-2}
\sum_{j=k_D}^{D-1}
P_jC_2^j
\frac{\Theta(t)^{j+1}}{(j+1)!},
\]
where
\[
P_j=\prod_{r=0}^{j-1}(1+rh).
\]
For Stuart--Landau, \(h=2\), so
\[
P_j
=
\prod_{r=0}^{j-1}(1+2r)
=
1\cdot3\cdot5\cdots(2j-1)
=
(2j-1)!!.
\]
Moreover,
\[
C_2=2,
\qquad
R=\frac12,
\qquad
R_+=1.
\]
Hence
\[
D C_2 R^{D+1}R_+^{p-2}
=
D\cdot 2\cdot 2^{-(D+1)}
=
\frac{D}{2^D}.
\]
Also,
\[
P_jC_2^j
=
(2j-1)!!\,2^j
=
\frac{(2j)!}{j!}.
\]
Therefore
\[
P_jC_2^j\frac{\theta^{j+1}}{(j+1)!}
=
\frac{(2j)!}{j!(j+1)!}\theta^{j+1}.
\]
The coefficient
\[
\frac{(2j)!}{j!(j+1)!}
\]
is the \(j\)-th Catalan number, denoted \(\mathrm{Cat}_j\). Thus the refined finite bound becomes
\begin{equation}\label{eq:SLfinite}
\boxed{
\epsilon_D(t)
\le
\frac{D}{2^D}
\sum_{j=k_D}^{D-1}
\mathrm{Cat}_j
\bigl(1-e^{-t}\bigr)^{j+1}.
}
\end{equation}
In particular, uniformly for \(t\in[0,T]\),
\begin{equation}\label{eq:SLfiniteT}
\boxed{
\sup_{t\in[0,T]}\epsilon_D(t)
\le
\frac{D}{2^D}
\sum_{j=k_D}^{D-1}
\mathrm{Cat}_j
\theta^{j+1},
\qquad
\theta:=1-e^{-T}.
}
\end{equation}

\paragraph{First-tail regime.}
The first geometric corollary uses
\[
\rho_T
=
hC_2\Theta(T)
=
2\cdot2\theta
=
4\theta.
\]
The first-tail condition is
\[
\rho_T<1,
\]
i.e.
\[
4(1-e^{-T})<1.
\]
Equivalently,
\[
1-e^{-T}<\frac14,
\]
so
\[
e^{-T}>\frac34,
\]
and hence
\[
\boxed{
T<\log\frac43\approx0.287682.
}
\]

In this regime,
\[
\gamma_1
=
hC_2R^h\Theta(T)
=
2\cdot2\cdot\left(\frac12\right)^2\theta
=
\theta.
\]
The constant is
\[
K_1
=
\frac{hC_2\Theta(T)R^2R_+^{2h-2}}{1-\rho_T}
=
\frac{4\theta\cdot(1/4)}{1-4\theta}
=
\frac{\theta}{1-4\theta}.
\]
Thus, for
\[
T<\log\frac43,
\]
we have
\begin{equation}\label{eq:SLfirst}
\boxed{
\sup_{t\in[0,T]}\epsilon_D(t)
\le
\frac{\theta}{1-4\theta}\,\theta^{k_D},
\qquad
\theta=1-e^{-T}.
}
\end{equation}
The first-tail geometric factor is therefore
\[
\boxed{
\gamma_1 = 1-e^{-T}.
}
\]

\paragraph{Last-tail regime.}
The second geometric corollary applies when
\[
\rho_T>1
\qquad\text{and}\qquad
R\rho_T<1.
\]
For Stuart--Landau,
\[
\rho_T=4\theta,
\qquad
R\rho_T=\frac12\cdot4\theta=2\theta.
\]
Thus the conditions are
\[
4\theta>1
\qquad\text{and}\qquad
2\theta<1.
\]
Equivalently,
\[
\frac14<\theta<\frac12.
\]
Since \(\theta=1-e^{-T}\), this gives
\[
\boxed{
\log\frac43<T<\log2.
}
\]

In this regime,
\[
\gamma_2
=
(R\rho_T)^h
=
(2\theta)^2
=
4\theta^2.
\]
The constant is
\[
K_2
=
\frac{hC_2\Theta(T)R R_+^{h-1}}{\rho_T-1}
=
\frac{4\theta\cdot(1/2)}{4\theta-1}
=
\frac{2\theta}{4\theta-1}.
\]
Therefore, for
\[
\log\frac43<T<\log2,
\]
we obtain
\begin{equation}\label{eq:SLlast}
\boxed{
\sup_{t\in[0,T]}\epsilon_D(t)
\le
\frac{2\theta}{4\theta-1}
\bigl(4\theta^2\bigr)^{k_D},
\qquad
\theta=1-e^{-T}.
}
\end{equation}
The last-tail geometric factor is
\[
\boxed{
\gamma_2
=
4(1-e^{-T})^2.
}
\]

\paragraph{Admissible horizon.}
Combining the two regimes, the refined bound proves convergence for every
\[
\theta<\frac12.
\]
Equivalently,
\[
1-e^{-T}<\frac12,
\]
that is
\[
\boxed{
T<\log2\approx0.693147.
}
\]
At \(T=\log2\), one has \(\theta=1/2\), and the last-tail factor becomes
\[
\gamma_2=4(1/2)^2=1,
\]
so the geometric estimate no longer proves convergence. For \(T>\log2\), the refined geometric estimate does not yield convergence.

The two closed-form corollaries have a singularity at the transition point
\[
\theta=\frac14,
\qquad
T=\log\frac43,
\]
because the corresponding geometric sums are estimated from opposite sides. This singularity is an artifact of the two simplified geometric estimates. The finite bound \eqref{eq:SLfiniteT} itself remains valid at \(T=\log(4/3)\).

\paragraph{Numerical illustration at \(T=0.2\).}
Let
\[
T=0.2.
\]
Then
\[
\theta=1-e^{-0.2}\approx0.181269.
\]
Since
\[
4\theta\approx0.725077<1,
\]
we are in the first-tail regime. The geometric factor is
\[
\gamma_1=\theta\approx0.181269.
\]
The constant is
\[
K_1
=
\frac{\theta}{1-4\theta}
\approx
0.659345.
\]

For example, with \(D=30\),
\[
k_D=\left\lceil\frac{30-2}{2}\right\rceil=14.
\]
The geometric corollary gives
\[
\sup_{t\in[0,0.2]}\epsilon_D(t)
\le
0.659345\cdot(0.181269)^{14}
\approx
2.73\times10^{-11}.
\]
The finite Catalan bound \eqref{eq:SLfiniteT} gives the sharper value
\[
\sup_{t\in[0,0.2]}\epsilon_D(t)
\le
\frac{30}{2^{30}}
\sum_{j=14}^{29}
\mathrm{Cat}_j(0.181269)^{j+1}
\approx
1.66\times10^{-12}.
\]

\paragraph{Numerical illustration at \(T=0.5\).}
Let
\[
T=0.5.
\]
Then
\[
\theta=1-e^{-0.5}\approx0.393469.
\]
Now
\[
4\theta\approx1.57388>1,
\qquad
2\theta\approx0.786939<1.
\]
Thus \(T=0.5\) lies in the last-tail regime. The geometric factor is
\[
\gamma_2
=
4\theta^2
\approx
0.619272.
\]
The constant is
\[
K_2
=
\frac{2\theta}{4\theta-1}
\approx
1.37127.
\]
For \(D=30\), again \(k_D=14\), and
\[
\sup_{t\in[0,0.5]}\epsilon_D(t)
\le
1.37127\cdot(0.619272)^{14}
\approx
1.67\times10^{-3}.
\]
The finite Catalan bound gives the sharper value
\[
\sup_{t\in[0,0.5]}\epsilon_D(t)
\le
\frac{30}{2^{30}}
\sum_{j=14}^{29}
\mathrm{Cat}_j(0.393469)^{j+1}
\approx
5.99\times10^{-5}.
\]

\paragraph{Comparison with the FP backward-integration bound.}
For comparison, FP Theorem 4.2 is applied after reducing the cubic Stuart--Landau system to a quadratic one. FP's theorem requires an a priori bound \(\alpha\) on the lifted trajectory and gives the bound labelled \(E_1(t)\) in their Theorem 4.2. :contentReference[oaicite:1]{index=1}

For Stuart--Landau, the quadratic lift satisfies
\[
\alpha=\frac12,
\qquad
\|\widetilde F_2\|_\infty=4,
\qquad
\mu_\infty(\widetilde F_1)=-1.
\]
Hence FP Theorem 4.2 gives
\[
E_1(t)
=
\frac12\left(2(1-e^{-t})\right)^N.
\]
Thus the FP backward-integration geometric factor is
\[
\boxed{
\alpha_{\mathrm{FP}}(t)=2(1-e^{-t})=2\theta.
}
\]
It proves convergence for
\[
2\theta<1,
\]
i.e.
\[
T<\log2.
\]
Thus the refined direct bound and FP Theorem 4.2 have the same admissible endpoint
\[
T=\log2.
\]
However, the refined direct factors are
\[
\gamma_1=\theta
\qquad
\text{for }\theta<1/4,
\]
and
\[
\gamma_2=4\theta^2=(2\theta)^2
\qquad
\text{for }1/4<\theta<1/2.
\]
Both are smaller than \(2\theta\) in their respective regimes. The interpretation must still account for the different truncation indices: here \(k_D\sim D/2\), whereas FP's \(N\) is the truncation order after quadratic lifting.

\subsection{Example: Van der Pol system}
We consider the Van der Pol system in the polynomial form
\begin{equation}\label{eq:VDP}
\begin{aligned}
\dot x_1 &= -\frac{x_1^3}{3}+x_1-x_2,\\
\dot x_2 &= x_1,
\end{aligned}
\end{equation}
with initial condition
\[
x_0=(1/2,0)^T.
\]
We take as observable
\[
g(x)=x_1.
\]
Thus \(d_g=1\) and \(\|v\|_1=1\).

The system is cubic, hence
\[
p=3,\qquad h:=p-1=2.
\]

The linear part is
\[
F_1=
\begin{pmatrix}
1 & -1\\
1 & 0
\end{pmatrix}.
\]
Therefore, in the \(\ell_\infty\)-logarithmic norm,
\[
\mu_1
=
\mu_\infty(F_1)
=
\max\{1+|-1|,\;0+|1|\}
=
2.
\]
Thus this example lies in the noncontractive case \(\mu_1>0\).

The nonlinear part is
\[
f_{\ge2}(x)
=
\begin{pmatrix}
-\frac{x_1^3}{3}\\
0
\end{pmatrix}.
\]
Hence
\[
C_2
=
\max_i\sum_{2\le |\gamma|\le 3}|c_{i\gamma}|
=
\frac13.
\]

Since \(\mu_1>0\), the sign-adapted time factor in the unified theorem is
\[
\Theta(T)
=
\frac{e^{h\mu_1T}-1}{h\mu_1}
=
\frac{e^{4T}-1}{4}.
\]
For convenience, write
\[
\theta:=\Theta(T)=\frac{e^{4T}-1}{4}.
\]

Since \(d_g=1\) and \(h=2\), the support-gap index is
\[
k_D
=
\left\lceil\frac{D-p+1}{p-1}\right\rceil\vee0
=
\left\lceil\frac{D-2}{2}\right\rceil\vee0.
\]
For \(D\ge2\),
\[
k_D=\left\lceil\frac{D-2}{2}\right\rceil.
\]

\paragraph{Trajectory bound.}
In contrast with the Stuart--Landau example, the Van der Pol system does not provide the same immediate monotone radial bound. We therefore use an a priori bound
\[
R\ge \sup_{t\in[0,T]}\|x(t;x_0)\|_\infty.
\]
For the numerical values below, one may use either a validated numerical enclosure or a simple analytic comparison estimate.

A convenient analytic estimate on short time intervals is obtained as follows. While \(x_2(t)\ge0\), one has
\[
\dot x_1
=
x_1-\frac{x_1^3}{3}-x_2
\le
x_1-\frac{x_1^3}{3}.
\]
Let \(y\) solve
\[
\dot y=y-\frac{y^3}{3},
\qquad
y(0)=\frac12.
\]
Then
\[
y(t)=\sqrt{\frac{3}{1+11e^{-2t}}},
\]
and scalar comparison gives
\[
x_1(t)\le y(t)
\]
on the interval under consideration. Moreover,
\[
x_2(t)=\int_0^t x_1(s)\,ds
\le
\int_0^t y(s)\,ds,
\]
with
\[
\int_0^t y(s)\,ds
=
\sqrt3
\left[
\operatorname{arsinh}\!\left(\frac{e^t}{\sqrt{11}}\right)
-
\operatorname{arsinh}\!\left(\frac{1}{\sqrt{11}}\right)
\right].
\]
For example, this yields the conservative analytic choices
\[
R=0.72\quad\text{on }[0,0.4],
\qquad
R=0.78\quad\text{on }[0,0.5].
\]
At \(T=0.2\), direct numerical integration gives the sharper value \(R\approx0.505\).

\paragraph{Explicit finite-\(D\) refined bound.}
The refined finite bound gives
\[
\epsilon_D(t)
\le
D C_2 R^{D+1}R_+^{p-2}
\sum_{j=k_D}^{D-1}
P_j C_2^j
\frac{\Theta(t)^{j+1}}{(j+1)!},
\]
where
\[
P_j=\prod_{r=0}^{j-1}(1+rh).
\]
For Van der Pol, \(h=2\), so
\[
P_j
=
\prod_{r=0}^{j-1}(1+2r)
=
1\cdot3\cdot5\cdots(2j-1)
=
(2j-1)!!.
\]
Moreover,
\[
C_2=\frac13.
\]
Hence
\[
P_j C_2^j
=
(2j-1)!!\,3^{-j}
=
\frac{(2j)!}{2^j j!}\,3^{-j}
=
\frac{(2j)!}{6^j j!}.
\]
Therefore
\[
P_j C_2^j\frac{\theta^{j+1}}{(j+1)!}
=
\frac{(2j)!}{j!(j+1)!}
\frac{\theta^{j+1}}{6^j}.
\]
The coefficient
\[
\frac{(2j)!}{j!(j+1)!}
\]
is the \(j\)-th Catalan number, denoted \(\mathrm{Cat}_j\). Thus the finite bound becomes
\begin{equation}\label{eq:VDPfinite}
\boxed{
\epsilon_D(t)
\le
\frac{D}{3}R^{D+1}R_+
\sum_{j=k_D}^{D-1}
\mathrm{Cat}_j
\frac{\Theta(t)^{j+1}}{6^j}.
}
\end{equation}
Equivalently,
\[
\epsilon_D(t)
\le
\frac{D}{3}R^{D+1}R_+
\Theta(t)
\sum_{j=k_D}^{D-1}
\mathrm{Cat}_j
\left(\frac{\Theta(t)}6\right)^j.
\]
In particular, uniformly for \(t\in[0,T]\),
\begin{equation}\label{eq:VDPfiniteT}
\boxed{
\sup_{t\in[0,T]}\epsilon_D(t)
\le
\frac{D}{3}R^{D+1}R_+
\theta
\sum_{j=k_D}^{D-1}
\mathrm{Cat}_j
\left(\frac{\theta}{6}\right)^j,
\qquad
\theta=\frac{e^{4T}-1}{4}.
}
\end{equation}

\paragraph{First-tail regime.}
The first geometric corollary uses
\[
\rho_T
=
hC_2\Theta(T)
=
2\cdot\frac13\theta
=
\frac{2\theta}{3}.
\]
The first-tail condition is
\[
\rho_T<1,
\]
that is
\[
\frac{2\theta}{3}<1.
\]
Since
\[
\theta=\frac{e^{4T}-1}{4},
\]
this condition is equivalent to
\[
\frac{e^{4T}-1}{6}<1,
\]
hence
\[
e^{4T}<7.
\]
Therefore
\[
\boxed{
T<\frac14\log 7\approx0.486477.
}
\]

In this regime,
\[
\gamma_1
=
hC_2R^h\Theta(T)
=
2\cdot\frac13 R^2\theta
=
\frac{2R^2\theta}{3}.
\]
Equivalently,
\[
\boxed{
\gamma_1
=
\frac{R^2}{6}\bigl(e^{4T}-1\bigr).
}
\]
The constant is
\[
K_1
=
\frac{hC_2\Theta(T)R^2R_+^{2h-2}}{1-\rho_T}
=
\frac{\frac23\theta R^2R_+^2}{1-\frac23\theta}.
\]
Thus, for
\[
T<\frac14\log 7,
\]
we have
\begin{equation}\label{eq:VDPfirst}
\boxed{
\sup_{t\in[0,T]}\epsilon_D(t)
\le
\frac{\frac23\theta R^2R_+^2}{1-\frac23\theta}
\left(\frac{2R^2\theta}{3}\right)^{k_D},
\qquad
\theta=\frac{e^{4T}-1}{4}.
}
\end{equation}

\paragraph{Last-tail regime.}
The second geometric corollary applies when
\[
\rho_T>1
\qquad\text{and}\qquad
R\rho_T<1.
\]
Here
\[
\rho_T=\frac{2\theta}{3}
=
\frac{e^{4T}-1}{6}.
\]
Thus the two conditions are
\[
\frac{e^{4T}-1}{6}>1
\]
and
\[
R\frac{e^{4T}-1}{6}<1.
\]
Equivalently,
\[
T>\frac14\log 7
\]
and
\[
T<\frac14\log\left(1+\frac6R\right).
\]
In this regime,
\[
\gamma_2
=
(R\rho_T)^h
=
\left(R\frac{e^{4T}-1}{6}\right)^2.
\]
The constant is
\[
K_2
=
\frac{hC_2\Theta(T)R R_+^{h-1}}{\rho_T-1}
=
\frac{\frac23\theta R R_+}{\frac23\theta-1}.
\]
Therefore, whenever
\[
\frac14\log 7<T<\frac14\log\left(1+\frac6R\right),
\]
we obtain
\begin{equation}\label{eq:VDPlast}
\boxed{
\sup_{t\in[0,T]}\epsilon_D(t)
\le
\frac{\frac23\theta R R_+}{\frac23\theta-1}
\left(
R\frac{e^{4T}-1}{6}
\right)^{2k_D},
\qquad
\theta=\frac{e^{4T}-1}{4}.
}
\end{equation}
The last-tail geometric factor is
\[
\boxed{
\gamma_2
=
\left(
R\frac{e^{4T}-1}{6}
\right)^2.
}
\]

\paragraph{Admissible horizon.}
For a fixed a priori trajectory bound \(R<1\), the two regimes together prove convergence whenever
\[
R\rho_T<1,
\]
that is
\[
R\frac{e^{4T}-1}{6}<1.
\]
Equivalently,
\[
\boxed{
T<\frac14\log\left(1+\frac6R\right).
}
\]
The transition between the two simplified geometric bounds occurs at
\[
\rho_T=1,
\]
that is
\[
T=\frac14\log7.
\]
As in the Stuart--Landau example, the simplified constants in the first-tail and last-tail estimates become singular at the transition point. This singularity is an artifact of the two geometric estimates. The finite bound \eqref{eq:VDPfiniteT} remains valid at the transition.

\paragraph{Numerical illustration at \(T=0.2\).}
Let
\[
T=0.2.
\]
Then
\[
\theta=\frac{e^{0.8}-1}{4}\approx0.306385,
\]
and
\[
\rho_T=\frac{2\theta}{3}\approx0.204257<1.
\]
Using the numerical trajectory bound from the draft,
\[
R=0.505,
\]
we are in the first-tail regime. The geometric factor is
\[
\gamma_1
=
\frac{2R^2\theta}{3}
\approx0.05209.
\]
The constant is
\[
K_1
=
\frac{\frac23\theta R^2}{1-\frac23\theta}
\approx0.06546.
\]

For \(D=30\),
\[
k_D=\left\lceil\frac{30-2}{2}\right\rceil=14.
\]
The geometric corollary gives
\[
\sup_{t\in[0,0.2]}\epsilon_D(t)
\le
0.06546\cdot(0.05209)^{14}
\approx
7.09\times10^{-20}.
\]
The finite Catalan bound \eqref{eq:VDPfiniteT} gives the sharper value
\[
\sup_{t\in[0,0.2]}\epsilon_D(t)
\le
\frac{30}{3}(0.505)^{31}
\sum_{j=14}^{29}
\mathrm{Cat}_j
\frac{(0.306385)^{j+1}}{6^j}
\approx
5.23\times10^{-21}.
\]

\paragraph{Numerical illustration at \(T=0.4\).}
Let
\[
T=0.4.
\]
Then
\[
\theta=\frac{e^{1.6}-1}{4}\approx0.988258,
\]
and
\[
\rho_T=\frac{2\theta}{3}\approx0.658839<1.
\]
Using the analytic bound
\[
R=0.72,
\]
we remain in the first-tail regime. The geometric factor is
\[
\gamma_1
=
\frac{2R^2\theta}{3}
\approx0.34154.
\]
The constant is
\[
K_1
=
\frac{\frac23\theta R^2}{1-\frac23\theta}
\approx1.00112.
\]
For \(D=30\), \(k_D=14\), hence
\[
\sup_{t\in[0,0.4]}\epsilon_D(t)
\le
1.00112\cdot(0.34154)^{14}
\approx
2.94\times10^{-7}.
\]
The finite Catalan bound gives
\[
\sup_{t\in[0,0.4]}\epsilon_D(t)
\le
\frac{30}{3}(0.72)^{31}
\sum_{j=14}^{29}
\mathrm{Cat}_j
\frac{(0.988258)^{j+1}}{6^j}
\approx
2.71\times10^{-8}.
\]

\paragraph{Numerical illustration at \(T=0.5\).}
Let
\[
T=0.5.
\]
Then
\[
\theta=\frac{e^2-1}{4}\approx1.597264,
\]
and
\[
\rho_T=\frac{2\theta}{3}\approx1.064843>1.
\]
Using the analytic bound
\[
R=0.78,
\]
we have
\[
R\rho_T\approx0.830577<1.
\]
Thus \(T=0.5\) lies in the last-tail regime. The last-tail geometric factor is
\[
\gamma_2
=
(R\rho_T)^2
\approx0.68986.
\]
The constant is
\[
K_2
=
\frac{\frac23\theta R}{\frac23\theta-1}
\approx12.8091.
\]
For \(D=30\), \(k_D=14\), and hence
\[
\sup_{t\in[0,0.5]}\epsilon_D(t)
\le
12.8091\cdot(0.68986)^{14}
\approx
7.08\times10^{-2}.
\]
The finite Catalan bound is substantially sharper:
\[
\sup_{t\in[0,0.5]}\epsilon_D(t)
\le
\frac{30}{3}(0.78)^{31}
\sum_{j=14}^{29}
\mathrm{Cat}_j
\frac{(1.597264)^{j+1}}{6^j}
\approx
2.48\times10^{-3}.
\]

\paragraph{Comparison with the FP backward-integration bound.}
For comparison, FP Theorem 4.2 can be applied after reducing the cubic Van der Pol system to a quadratic one. The lifted quadratic system has
\[
\mu_\infty(\widetilde F_1)=4,
\qquad
\|\widetilde F_2\|_\infty=\frac23.
\]
If \(R<1\), then the lifted trajectory satisfies
\[
\|\widetilde x(t)\|_\infty
=
\max\{\|x(t)\|_\infty,\|x(t)\|_\infty^2\}
\le R.
\]
Thus FP Theorem 4.2 gives the geometric factor
\[
\alpha_{\mathrm{FP}}(T)
=
\frac{R\|\widetilde F_2\|_\infty}{\mu_\infty(\widetilde F_1)}
\left(e^{\mu_\infty(\widetilde F_1)T}-1\right)
=
\frac{R}{6}\bigl(e^{4T}-1\bigr).
\]
Equivalently,
\[
\alpha_{\mathrm{FP}}(T)=R\rho_T.
\]
Therefore, in the first-tail regime,
\[
\gamma_1
=
R^2\rho_T
=
R\,\alpha_{\mathrm{FP}}(T),
\]
while in the last-tail regime,
\[
\gamma_2
=
(R\rho_T)^2
=
\alpha_{\mathrm{FP}}(T)^2.
\]
Thus, when \(R<1\), the refined direct geometric factors are smaller than the FP backward-integration factor. As in the Stuart--Landau example, the interpretation must still account for the different truncation indices: the direct estimate uses \(k_D\sim D/2\), whereas the FP estimate is expressed in terms of the truncation order of the lifted quadratic system.
\fi

\section{Proofs}\label{sec:proofs}
Throughout this section we will assume the notation and set up introduced in subsection \ref{sub:setup}.

\subsection{Elementary degree and norm estimates}
\ifmai
For a row vector \(r\in\R^{1\times M}\), define
\[
  \deg_{\max}(r):=\max\{\deg(\alpha_j):r_j\ne0\},
\]
with the convention \(\deg_{\max}(0)=-\infty\), and
\[
  \deg_{\min}(r):=\min\{\deg(\alpha_j):r_j\ne0\},
\]
with the convention \(\deg_{\min}(0)=+\infty\).

\begin{lemma}[Degree preservation and degree-local bounds for \(A_1\)]
\label{lem:degree-local-a1}
Let \(\mathcal H_m\) denote the subspace spanned by monomials of total degree
\(m\). The matrix \(A_1\) is block diagonal with respect to the decomposition
\[
  \R^M=\mathcal H_1\oplus\cdots\oplus \mathcal H_D .
\]
Consequently, for every \(\theta\ge0\), also \(e^{A_1\theta}\) is block
diagonal with respect to the same decomposition. In particular, if a row vector
\(r\) is supported only on degrees belonging to a set \(I\subseteq\{1,\ldots,D\}\),
then \(r e^{A_1\theta}\) is supported only on the same set of degrees.

Moreover, if \(r\) is supported on degrees at least \(a\), then, for
\(\mu_1<0\),
\[
  \|r e^{A_1\theta}\|_1
  \le
  e^{a\mu_1\theta}\|r\|_1.
\]
If \(r\) is supported on degrees at most \(b\), then, for \(\mu_1\ge0\),
\[
  \|r e^{A_1\theta}\|_1
  \le
  e^{b\mu_1\theta}\|r\|_1.
\]
\end{lemma}
\begin{proof}
Let \(x^\alpha\) be a monomial of total degree \(m=|\alpha|\). Since
\(A_1\) is induced by the linear vector field \(F_1x\),
\[
  \mathcal L_{F_1x}(x^\alpha)
  =
  \sum_{i=1}^n \alpha_i x^{\alpha-e_i}(F_1x)_i
  =
  \sum_{i,\ell=1}^n \alpha_i (F_1)_{i\ell}x^{\alpha-e_i+e_\ell}.
\]
Every nonzero monomial \(x^{\alpha-e_i+e_\ell}\) appearing in this expression
has total degree
\[
  |\alpha-e_i+e_\ell|=|\alpha|=m.
\]
Thus \(A_1\) does not mix homogeneous degrees. Equivalently, after ordering
monomials by total degree,
\[
  A_1=\operatorname{diag}(A_{1,1},A_{1,2},\ldots,A_{1,D}),
\]
where \(A_{1,m}\) is the block acting on \(\mathcal H_m\). Hence
\[
  e^{A_1\theta}
  =
  \operatorname{diag}\bigl(
  e^{A_{1,1}\theta},
  e^{A_{1,2}\theta},
  \ldots,
  e^{A_{1,D}\theta}
  \bigr),
\]
so \(e^{A_1\theta}\) preserves the same degree blocks.

For the norm estimates, decompose a row vector \(r\) as
\[
  r=\sum_{m=1}^D r_m,
\]
where \(r_m\) is supported on degree \(m\). Then
\[
  r e^{A_1\theta}
  =
  \sum_{m=1}^D r_m e^{A_{1,m}\theta}.
\]
The logarithmic-norm estimate for the degree-\(m\) block gives
\[
  \mu_\infty(A_{1,m})\le m\mu_1,
\]
and hence
\[
  \|e^{A_{1,m}\theta}\|_\infty
  \le
  e^{m\mu_1\theta}.
\]
Using
\[
  \|qM\|_1\le \|q\|_1\|M\|_\infty
\]
for row vectors, we obtain
\[
  \|r_m e^{A_{1,m}\theta}\|_1
  \le
  \|r_m\|_1 e^{m\mu_1\theta}.
\]

If \(\mu_1<0\) and \(r\) is supported on degrees at least \(a\), then
\(m\ge a\) on every nonzero block, so
\[
  e^{m\mu_1\theta}\le e^{a\mu_1\theta}.
\]
Therefore
\[
  \|r e^{A_1\theta}\|_1
  \le
  \sum_{m\ge a}\|r_m\|_1 e^{m\mu_1\theta}
  \le
  e^{a\mu_1\theta}\sum_{m\ge a}\|r_m\|_1
  =
  e^{a\mu_1\theta}\|r\|_1.
\]

If \(\mu_1\ge0\) and \(r\) is supported on degrees at most \(b\), then
\(m\le b\) on every nonzero block, so
\[
  e^{m\mu_1\theta}\le e^{b\mu_1\theta}.
\]
Thus
\[
  \|r e^{A_1\theta}\|_1
  \le
  e^{b\mu_1\theta}\|r\|_1.
\]
\end{proof}
\fi
\begin{lemma}[Degree structure and row bounds]
\label{lem:degree-structure}
The following statements hold.
\begin{enumerate}[label=\textup{(\roman*)}]
\item Both  \(A_1\) and $e^{A_1\theta}$ ($\theta\geq 0$) are block diagonal with respect to total degree, more precisely
we have
  $A_1 =\operatorname{diag}(A_{1,1},A_{1,2},\ldots,A_{1,D})$ and
   $e^{A_1\theta} =\operatorname{diag}(e^{A_{1,1}\theta},e^{A_{1,2}\theta},\ldots,e^{A_{1,D}\theta})
$,
where \(A_{1,m}\) is the block of elements $(A_1)_{i,j}$ such that $\deg(\alpha_i)=\deg(\alpha_j)=m$.  In particular, for any row vector $q$, $q$ and $qe^{A_1\theta}$ are supported on the same degrees.
 Moreover
      \[
        \mu_\infty(A_{1,m})\le m\mu_1.
      \]
\item If \(\alpha_j=x^\beta\) has degree \(|\beta|=m\), then  \(e_j^TA_2\) is supported on degrees in
      $
        \{m+1,m+2,\ldots,\min(D,m+h)\}
      $. Moreover
      \[
        \|e_j^TA_2\|_1\le mC_2.
      \]
      Consequently, if \(q\) is a row supported on degrees at most
      \(b\), then
      \[
        \|q A_2\|_1\le \|q \|_1 bC_2.
      \]
\item 
       \(w(\tau)=B\psi(x(\tau;x_0))\)  is supported only on
      degrees at least
     $
        D-h+1=D-p+2,
     $. Moreover,   for all \(\tau\in[0,T]\),
      \[
        \|w(\tau)\|_\infty\le DC_2R^{D+1}R_+^{h-1}.
      \]
\end{enumerate}
\end{lemma}
\begin{proof}
\begin{itemize}
\item[(i)]
The block structure of $A_1$ follows by the definition of $A$ and of $A_1$. As $(A_1\theta)^k=\operatorname{diag}((A_{1,1}\theta)^k,(A_{1,2}\theta)^k,\ldots,(A_{1,D}\theta)^k)$, the exponential matrix $e^{A_1\theta}=\sum_{k\geq 0} \frac{(A_1\theta)^k}{k!}$ preserves this structure.
Now, consider  any nonzero row vector $q$ and decompose it as $q=\sum_{m=1}^D q_m$, where $q_m$ is the $m$-degree component of $q$. Since $e^{A_1\theta}$ is   nonsingular, whenever $q_m\neq 0$ we have that $q_m e^{A_1\theta}\neq 0$ and supported on degree $m$: hence   $qe^{A_1\theta}=\sum_{m=1}^D q_m e^{A_1\theta}$ is supported exactly on the same degrees as $q$.
Next, let \(\alpha_j=x^\beta\), with \(|\beta|=m\).  The entries of   the row $e^T_j A_1$ are determined by  the linear vector field \(F_1 x\),
\[
  \Lie_{F_1x}(x^\beta)
  =\sum_{i=1}^n \beta_i x^{\beta-e_i}(F_1x)_i
  =\sum_{i=1}^n\sum_{\ell=1}^n
    \beta_i(F_1)_{i\ell}x^{\beta-e_i+e_\ell}
\]
(NB: terms with exponents having a negative $i$-th component are 0, as $\beta_i=0$).
In the above sum, the exponent is $\beta$ precisely when $i=\ell$, so the diagonal coefficient in row $e^T_j A_1$ is
\[
  \sum_{i=1}^n \beta_i(F_1)_{ii},
\]
while the off-diagonal row sum is bounded by
\[
  \sum_{i=1}^n\beta_i\sum_{\ell\ne i}|(F_1)_{i\ell}|.
\]
Therefore the logarithmic-norm contribution of this row is bounded by
\[
  \sum_{i=1}^n \beta_i
  \left((F_1)_{ii}+\sum_{\ell\ne i}|(F_1)_{i\ell}|\right)
  \le |\beta|\mu_1=m\mu_1.
\]
Taking the maximum over rows in the degree-\(m\) block gives
\(\mu_\infty(A_{1,m})\le m\mu_1\). This proves (i).

\item[(ii)]
The entries of   the row $e^T_j A_2$ are determined by  the nonlinear vector field,
\[
  \Lie_{f_{\ge2}}(x^\beta)
  =\sum_{i=1}^n \beta_i x^{\beta-e_i}(f_{\ge2})_i(x)
  =\sum_{i=1}^n\sum_{2\le |\gamma|\le p}
    \beta_i c_{i\gamma}x^{\beta-e_i+\gamma}.
\]
Every nonzero monomial has total degree
$
  |\beta|-1+|\gamma|\in\{m+1,\ldots,m+h\}
$.
Out of these,  only monomials up to total degree $D$ are considered in $A$, so $e^T_j A_2$ is supported on degrees in
\[
  |\beta|-1+|\gamma|\in\{m+1,\ldots,\min(D,m+h)\}.
\]
The absolute row sum is bounded by
\[
  \sum_{i=1}^n\beta_i\sum_{2\le |\gamma|\le p}|c_{i\gamma}|
  \le |\beta|C_2=mC_2.
\]
Thus, if \(q\) is supported on   degrees at most \(b\), then
\[
  \|qA_2\|_1
  \le \sum_j |q_j|\,\|e_j^TA_2\|_1
  \le bC_2\sum_j |q_j|
  =bC_2\|q\|_1.
\]
This proves (ii).

\item[(iii)]
By definition, the vector \(w(\tau)=B\psi(x(\tau;x_0))\) contains exactly the overflow terms of
\(\Lie_{f_{\ge2}}(\alpha)\), namely the  Lie-derivative nonlinear terms
of degree is    \(>D\). Since $\Lie_f$ can raise degree by
at most \(h\),  a monomial of degree \(m\) contributes to $w(\tau)$ only if
\[
  m+h>D,
\]
that is, only if \(m\ge D-h+1\). This proves the support assertion.
For a row $j$  of degree \(m\le D\), the 1-norm of the $j$-th row of $B$
is bounded as \(\|e^T_j B \|_1 \leq mC_2\le DC_2\). The entries of \(\psi\) are
monomials of degrees between \(D+1\) and \(D+h\). Since
\(\|x(\tau;x_0)\|_\infty\le R\), each entry of $\psi(x(\tau;x_0))$ is bounded by
\[
  R^{D+1}R_+^{h-1}.
\]
Indeed, for \(R\ge1\) this bounds all degrees up to \(D+h\), and for
\(0<R<1\) the largest value occurs at the smallest degree \(D+1\). Therefore
\[
  \|w(\tau)\|_\infty\le DC_2R^{D+1}R_+^{h-1}.
\]
This proves (iii).
\end{itemize}
\end{proof}

\subsection{Variation of parameters and Dyson expansion}
\begin{lemma}[Variation of parameters]
\label{lem:voc}
For every \(t\in[0,T]\),
\[
  \epsilon_D(t;x_0)
  =\int_0^t v^Te^{A(t-\tau)}w(\tau)\,d\tau.
\]
\end{lemma}
\begin{proof}
Let
\[
  z(t):=\alpha(x(t;x_0))\quad z_D(t):=e^{At}\alpha(x_0)\quad e(t):=z(t)-z_D(t)\,.
\]
By construction
\ifmai
 of the truncated Carleman representation,
\[
  \dot z(t)=Az(t)+w(t),
  \qquad z(0)=\alpha(x_0).
\]
The truncated Carleman solution satisfies
\[
  \dot z_D(t)=Az_D(t),
  \qquad z_D(0)=\alpha(x_0).
\]
Set \(e(t):=z(t)-z_D(t)\). Then
\fi
the error  vector $e(t)$ satisfies
\[
  \dot e(t)=Ae(t)+w(t),
  \qquad e(0)=0.
\]
The finite-dimensional variation-of-parameters formula  \cite[Prop.2.67]{Chicone} gives
\[
  e(t)=\int_0^t e^{A(t-\tau)}w(\tau)\,d\tau.
\]
Multiplying by \(v^T\) yields the claim.
\end{proof}

The next two results are standard, see e.g. \cite{Dyson,Pazy1983}; we report the simple proofs to be self-contained.

\begin{lemma}[Duhamel formula]\label{lemma:Duhamel}
Let \(A=A_1+A_2\) be square matrices. Then, for every \(s\ge0\),
\[
  e^{As}
  =
  e^{A_1s}
  +
  \int_0^s e^{A_1(s-r)}A_2e^{Ar}\,dr .
\]
\end{lemma}
\begin{proof}
Fix \(s\ge0\) and define
\[
  \Phi(r):=e^{A_1(s-r)}e^{Ar},
  \qquad 0\le r\le s .
\]
Then
\[
  \Phi'(r)
  =
  -e^{A_1(s-r)}A_1e^{Ar}
  +
  e^{A_1(s-r)}Ae^{Ar}.
\]
Since \(A=A_1+A_2\), this becomes
\[
  \Phi'(r)
  =
  e^{A_1(s-r)}(A-A_1)e^{Ar}
  =
  e^{A_1(s-r)}A_2e^{Ar}.
\]
Integrating over \(r\in[0,s]\), we get
\[
  \Phi(s)-\Phi(0)
  =
  \int_0^s e^{A_1(s-r)}A_2e^{Ar}\,dr .
\]
But
\[
  \Phi(s)=e^{As},
  \qquad
  \Phi(0)=e^{A_1s}.
\]
Therefore
\[
  e^{As}-e^{A_1s}
  =
  \int_0^s e^{A_1(s-r)}A_2e^{Ar}\,dr,
\]
which proves the formula.
\end{proof}

\begin{lemma}[Dyson expansion]
\label{lem:dyson}
For every \(s\ge0\),
\[
  e^{As}
  =
  \sum_{j=0}^{\infty}
  \int_{\Delta_j(s)}
  e^{A_1(s-s_j)}A_2e^{A_1(s_j-s_{j-1})}
  \cdots A_2e^{A_1s_1}
  \,ds_1\cdots ds_j,
\]
where
\[
  \Delta_j(s):=\{(s_1,\ldots,s_j):0\le s_1\le\cdots\le s_j\le s\}
\]
for \(j\ge1\), and \(\Delta_0(s):=\{()\}\). The \(j=0\) term is
understood as \(e^{A_1s}\). The series converges absolutely in any matrix norm,
uniformly for \(s\) in compact intervals.
\end{lemma}

\begin{proof}
By Lemma~\ref{lemma:Duhamel}, Duhamel's formula gives
\[
  e^{As}
  =
  e^{A_1s}
  +
  \int_0^s e^{A_1(s-r)}A_2e^{Ar}\,dr .
\]
Define
\[
  U_0(s):=e^{A_1s},
  \qquad
  U_{j+1}(s):=
  \int_0^s e^{A_1(s-r)}A_2U_j(r)\,dr .
\]
Iterating Duhamel's formula yields, for every \(N\ge0\),
\[
  e^{As}
  =
  \sum_{j=0}^{N}U_j(s)+R_{N+1}(s),
\]
where
\[
  R_{N+1}(s)
  =
  \int_{0\le r_{N+1}\le\cdots\le r_1\le s}
  e^{A_1(s-r_1)}A_2e^{A_1(r_1-r_2)}A_2\cdots
  A_2e^{A_1(r_N-r_{N+1})}A_2e^{Ar_{N+1}}
  \,dr_{N+1}\cdots dr_1 .
\]
We first show that \(R_{N+1}(s)\to0\) uniformly on compact time intervals.
Fix \(S>0\) and \(s\in[0,S]\). In any submultiplicative matrix norm,
\[
\begin{aligned}
  \|R_{N+1}(s)\|
  &\le
  e^{\|A_1\|S}e^{\|A\|S}
  \|A_2\|^{N+1}
  \operatorname{vol}\{0\le r_{N+1}\le\cdots\le r_1\le s\}  \\
  &=
  e^{(\|A_1\|+\|A\|)S}
  \frac{(s\|A_2\|)^{N+1}}{(N+1)!}  \\
  &\le
  e^{(\|A_1\|+\|A\|)S}
  \frac{(S\|A_2\|)^{N+1}}{(N+1)!}.
\end{aligned}
\]
Hence \(R_{N+1}(s)\to0\) uniformly for \(s\in[0,S]\).
Expanding the recursion defining \(U_j\), and then changing variables by
\[
  s_j=r_1,\quad s_{j-1}=r_2,\quad\ldots,\quad s_1=r_j,
\]
gives
\[
  U_j(s)
  =
  \int_{\Delta_j(s)}
  e^{A_1(s-s_j)}A_2e^{A_1(s_j-s_{j-1})}
  \cdots A_2e^{A_1s_1}
  \,ds_1\cdots ds_j .
\]
Letting \(N\to\infty\) gives the stated expansion.
It remains only to record absolute convergence. For \(s\in[0,S]\),
\[
\left\|
e^{A_1(s-s_j)}
A_2 e^{A_1(s_j-s_{j-1})}
\cdots
A_2 e^{A_1s_1}
\right\|
\le
e^{\|A_1\|S}\|A_2\|^j .
\]
Since
\[
  \operatorname{vol}(\Delta_j(s))=\frac{s^j}{j!},
\]
the norm of the \(j\)-th term is bounded by
\[
  e^{\|A_1\|S}\frac{(S\|A_2\|)^j}{j!}.
\]
The corresponding majorant series is convergent. Therefore the Dyson series
converges absolutely and uniformly for \(s\in[0,S]\) in every submultiplicative
matrix norm. Since all norms are equivalent in finite dimension, the same
conclusion holds in any matrix norm.
\end{proof}

\subsection{The refined Dyson-tail estimate}
\begin{lemma}[Degree-local $e^{A_1\theta}$ estimate]
\label{lem:degree-local-a1}
Let \(q\in\R^{1\times M}\). If \(\mu_1<0\) and  $q$ is supported   on   degrees at least
\(a\)   then for every \(\theta\ge0\),
\[
  \|q e^{A_1\theta}\|_1\le e^{a\mu_1\theta}\|q\|_1.
\]
If \(\mu_1\ge 0\) and \(q\) is supported   on degrees at most
\(b\), then
\[
  \|q e^{A_1\theta}\|_1\le e^{b\mu_1\theta}\|q\|_1.
\]
\end{lemma}
\begin{proof}
Write \(q=\sum_m q_m\), where \(q_m\) is the component of \(r\) supported on
  degree \(m\). By Lemma~\ref{lem:degree-structure},
\(A_1\) is block diagonal with degree-\(m\) block \(A_{1,m}\), and
\(\mu_\infty(A_{1,m})\le m\mu_1\). Hence
\[
  \|e^{A_{1,m}\theta}\|_\infty\le e^{m\mu_1\theta}.
\]
Using \(\|qU\|_1\le\|q\|_1\|U\|_\infty\) for any matrix $U$:
\[
  \|q e^{A_1\theta}\|_1
  \le \sum_m \|q_m\|_1e^{m\mu_1\theta}.
\]
If \(\mu_1<0\) and \(m\ge a\), then
\(e^{m\mu_1\theta}\le e^{a\mu_1\theta}\). Therefore
\[
  \|q e^{A_1\theta}\|_1
  \le e^{a\mu_1\theta}\sum_m\|q_m\|_1
  =e^{a\mu_1\theta}\|q\|_1.
\]
If \(\mu_1\ge0\) and \(m\le b\), then
\(e^{m\mu_1\theta}\le e^{b\mu_1\theta}\), and similarly
\[
  \|q e^{A_1\theta}\|_1
  \le e^{b\mu_1\theta}\|q\|_1.
\]
\end{proof}

\begin{lemma}[Refined Dyson-tail estimate]
\label{lem:refined-tail}
Let \(w\in\R^M\) be supported   on  degrees at least
\(D-h+1\), and let \(v\) be supported only on degree 1. Then, for
all \(s\in[0,T]\),
\[
  |v^Te^{As}w|
  \le
  \|v\|_1\|w\|_\infty e^{\mu_1s}
  \sum_{j=k_D}^{D-1}P_jC_2^j\frac{\Theta(s)^j}{j!}.
\]
\end{lemma}
\begin{proof}
By the absolute convergence in Lemma~\ref{lem:dyson}, and since
\(U\mapsto v^TUw\) is a continuous linear functional, we may apply \(v^T\) and
\(w\) term by term to the Dyson series:
\begin{equation}\label{eq:lemdyson}
  v^Te^{As}w
  =
  \sum_{j=0}^{\infty}
  \int_{\Delta_j(s)}
  v^Te^{A_1(s-s_j)}
  A_2e^{A_1(s_j-s_{j-1})}
  \cdots
  A_2e^{A_1s_1}w
  \,ds_1\cdots ds_j .
\end{equation}
We shall use the following two elementary norm estimates. If \(q\) is a row
vector, \(U\) a matrix, and \(u\) a column vector, then
\[
  |qu|\le \|q\|_1\|u\|_\infty,
  \qquad
  \|qU\|_1\le \|q\|_1\|U\|_\infty .
\]
The first estimate is \(\ell^1\)-\(\ell^\infty\) duality. The second is the
corresponding induced estimate for row vectors, equivalently applied to
\(U^Tq^T\).

First consider the expression inside the integrand. The row vector \(v^T\) is supported on
degree \(1\). By Lemma~\ref{lem:degree-structure}(i), each factor
\(e^{A_1\theta}\) leaves the support  degree unchanged. By Lemma~\ref{lem:degree-structure}(ii), each   occurrence
of \(A_2\) increases  the support  degree   by at least \(1\) and at most \(h\), up to a maximum degree of $D$. Therefore,
after \(j\) occurrences of \(A_2\), the row support is contained in degrees
between
\begin{equation}\label{eq:range}
  1+j
  \qquad\text{and}\qquad
  \min(D,1+jh).
\end{equation}
If \(j<k_D\), then
\[
  1+jh<D-h+1,
\]
so the row support is disjoint from the support of \(w\), which by hypothesis
is contained in degrees at least \(D-h+1\). Hence all Dyson terms in \eqref{eq:lemdyson} with
\(j<k_D\) vanish. If \(j\ge D\), then \eqref{eq:range} gives  an empty range,  hence these terms vanish as well. Thus only
indices $j$ such that
\[
  k_D\le j\le D-1
\]
can contribute to the sum in \eqref{eq:lemdyson}.
\ifmai
Fix such a \(j\), and fix
\((s_1,\ldots,s_j)\in\Delta_j(s)\). Set
\[
  \ell_0:=s-s_j,\qquad
  \ell_r:=s_{j-r+1}-s_{j-r}\quad(1\le r\le j-1),
  \qquad
  \ell_j:=s_1 .
\]
Then the fixed \(j\)-th integrand can be written as
\[
  v^T e^{A_1\ell_0}A_2e^{A_1\ell_1}
  A_2\cdots A_2e^{A_1\ell_j}w .
\]
%
Before the \(r\)-th occurrence of \(A_2\), with \(r=0,\ldots,j-1\), the row
support has degree at most \(1+rh\). By Lemma~\ref{lem:degree-structure},
the corresponding \(A_2\)-factor satisfies the row-vector estimate
\[
  \|qA_2\|_1
  \le
  (1+rh)C_2\|q\|_1
\]
for any row vector \(q\) with this degree support. Therefore the product of
the \(j\) nonlinear factors contributes at most
\[
  \prod_{r=0}^{j-1}(1+rh)C_2
  =
  P_jC_2^j
\]
to the propagated row \(\ell^1\)-norm.

It remains to estimate the \(A_1\)-semigroup factors. If \(\mu_1<0\), then
after \(r\) occurrences of \(A_2\) the row support is contained in degrees at
least \(1+r\). Applying Lemma~\ref{lem:degree-local-a1} to the \(A_1\)-factor
over the interval \(\ell_r\) gives the factor
\[
  e^{(1+r)\mu_1\ell_r}.
\]
Multiplying over \(r=0,\ldots,j\), the total \(A_1\)-contribution is bounded by
\[
  \exp\left(\mu_1\sum_{r=0}^j(1+r)\ell_r\right).
\]
A direct telescoping computation gives
\[
  \sum_{r=0}^j(1+r)\ell_r
  =
  s+s_1+\cdots+s_j .
\]
Hence, when \(\mu_1<0\), the \(A_1\)-contribution is bounded by
\[
  e^{\mu_1s}e^{\mu_1(s_1+\cdots+s_j)}.
\]

If \(\mu_1\ge0\), then after \(r\) occurrences of \(A_2\) the row support is
contained in degrees at most \(1+rh\). Applying
Lemma~\ref{lem:degree-local-a1} to the \(A_1\)-factor over \(\ell_r\) gives
the factor
\[
  e^{(1+rh)\mu_1\ell_r}.
\]
Multiplying over \(r=0,\ldots,j\), the total \(A_1\)-contribution is bounded by
\[
  \exp\left(\mu_1\sum_{r=0}^j(1+rh)\ell_r\right).
\]
Again by telescoping,
\[
  \sum_{r=0}^j(1+rh)\ell_r
  =
  s+h(s_1+\cdots+s_j).
\]
Hence, when \(\mu_1\ge0\), the \(A_1\)-contribution is bounded by
\[
  e^{\mu_1s}e^{h\mu_1(s_1+\cdots+s_j)}.
\]

By the definition of \(\sigma\), the two cases are expressed uniformly as
\[
  e^{\mu_1s}e^{\sigma\mu_1(s_1+\cdots+s_j)}.
\]
Combining the \(A_1\)- and \(A_2\)-bounds, and finally applying
\(\ell^1\)-\(\ell^\infty\) duality to the terminal pairing with \(w\), the
absolute value of the fixed \(j\)-th Dyson integrand is bounded by
\[
  \|v\|_1\|w\|_\infty
  P_jC_2^j
  e^{\mu_1s}e^{\sigma\mu_1(s_1+\cdots+s_j)} .
\]

Integrating over \(\Delta_j(s)\), we obtain
\[
  \int_{\Delta_j(s)}
  e^{\sigma\mu_1(s_1+\cdots+s_j)}
  \,ds_1\cdots ds_j .
\]
The integrand is symmetric in \(s_1,\ldots,s_j\), and the ordered simplices
partition \([0,s]^j\) up to measure-zero boundaries. Hence
\[
  \int_{\Delta_j(s)}
  e^{\sigma\mu_1(s_1+\cdots+s_j)}
  \,ds_1\cdots ds_j
  =
  \frac1{j!}
  \int_{[0,s]^j}
  e^{\sigma\mu_1(u_1+\cdots+u_j)}
  \,du_1\cdots du_j .
\]
By product structure,
\[
  \int_{[0,s]^j}
  e^{\sigma\mu_1(u_1+\cdots+u_j)}
  \,du_1\cdots du_j
  =
  \left(\int_0^s e^{\sigma\mu_1u}\,du\right)^j
  =
  \Theta(s)^j .
\]
Therefore,
\[
  \int_{\Delta_j(s)}
  e^{\sigma\mu_1(s_1+\cdots+s_j)}
  \,ds_1\cdots ds_j
  =
  \frac{\Theta(s)^j}{j!}.
\]

Summing over \(j=k_D,\ldots,D-1\), we get
\[
  |v^Te^{As}w|
  \le
  \|v\|_1\|w\|_\infty e^{\mu_1s}
  \sum_{j=k_D}^{D-1}
  P_jC_2^j\frac{\Theta(s)^j}{j!}.
\]
This proves the claim.
\end{proof}
\begin{proof}
By Lemma~\ref{lem:dyson},
\[
  v^Te^{As}w
  =
  \sum_{j=0}^{\infty}
  \int_{\Delta_j(s)}
  v^Te^{A_1(s-s_j)}
  A_2e^{A_1(s_j-s_{j-1})}
  \cdots
  A_2e^{A_1s_1}w
  \,ds_1\cdots ds_j .
\]
We first record the two norm estimates used throughout. If \(r\) is a row
vector and \(u\) is a column vector, then
\[
  |ru|\le \|r\|_1\|u\|_\infty .
\]
If \(M\) is a matrix, then
\[
  \|rM\|_1\le \|r\|_1\|M\|_\infty .
\]
Equivalently, this is the usual induced-norm inequality applied to
\(M^Tr^T\). Thus the row vectors generated along a Dyson path are estimated in
\(\ell^1\), while the final vector \(w\) is estimated in \(\ell^\infty\).

Consider first the support restrictions. The row vector \(v^T\) is supported on
degree-one monomials. By Lemma~\ref{lem:degree-local-a1}, \(e^{A_1\theta}\)
does not mix homogeneous degrees. By Lemma~\ref{lem:degree-structure}, each
nonzero occurrence of \(A_2\) raises total degree by at least \(1\) and at most
\(h\). Hence after \(j\)
occurrences of \(A_2\), the row support is contained in degrees between
\[
  1+j
  \qquad\text{and}\qquad
  1+jh .
\]
If \(j<k_D\), then
\[
  1+jh<D-h+1,
\]
so the row support is disjoint from the support of \(w\). Hence all terms with
\(j<k_D\) vanish. If \(j\ge D\), then the minimum possible degree after
\(j\) nonzero occurrences of \(A_2\) is at least \(1+j\ge D+1\), outside the
truncated degree-\(D\) space. Hence all terms with \(j\ge D\) vanish. Thus only
\[
  k_D\le j\le D-1
\]
can contribute.
\fi
Fix such a \(j\), and fix
\[
  (s_1,\ldots,s_j)\in\Delta_j(s).
\]
Set
\[
  \ell_0:=s-s_j,\qquad
  \ell_r:=s_{j-r+1}-s_{j-r}\quad(1\le r\le j-1),
  \qquad
  \ell_j:=s_1 .
\]
Thus the \(j\)-th Dyson integrand is
\[
  v^Te^{A_1\ell_0}A_2e^{A_1\ell_1}
  A_2\cdots A_2e^{A_1\ell_j}w .
\]
Define row vectors \(q_r,\widetilde q_r\) recursively by
\[
  q_0:=v^T,
\]
and, for \(r=0,\ldots,j-1\),
\[
  \widetilde q_r:=q_r e^{A_1\ell_r},
  \qquad
  q_{r+1}:=\widetilde q_r A_2.
\]
Finally set
\[
  \widetilde q_j:=q_j e^{A_1\ell_j}.
\]
Then the fixed \(j\)-th integrand is exactly
\[
  \widetilde q_j w .
\]
We now estimate \(\|\widetilde q_j\|_1\) factor by factor. After \(r\)
occurrences of \(A_2\), the row vector \(q_r\) is supported only on degrees
between
\[
  1+r
  \qquad\text{and}\qquad
  1+rh .
\]
For the \(A_1\)-factor over the interval \(\ell_r\), Lemma~\ref{lem:degree-local-a1}
gives the following estimate. If \(\mu_1<0\), the lower degree bound \(1+r\)
is used, and
\[
  \|\widetilde q_r\|_1
  =
  \|q_re^{A_1\ell_r}\|_1
  \le
  e^{(1+r)\mu_1\ell_r}\|q_r\|_1 .
\]
If \(\mu_1\ge0\), the upper degree bound \(1+rh\) is used, and
\[
  \|\widetilde q_r\|_1
  =
  \|q_re^{A_1\ell_r}\|_1
  \le
  e^{(1+rh)\mu_1\ell_r}\|q_r\|_1 .
\]
For the subsequent \(A_2\)-factor, with \(r=0,\ldots,j-1\), the row
\(\widetilde q_r\) is still supported on degrees at most \(1+rh\), since
\(A_1\) does not mix degrees. Hence Lemma~\ref{lem:degree-structure} gives
\[
  \|q_{r+1}\|_1
  =
  \|\widetilde q_rA_2\|_1
  \le
  (1+rh)C_2\|\widetilde q_r\|_1 .
\]
Combining these two estimates along the path gives the following. If
\(\mu_1<0\), then
\[
  \|\widetilde q_j\|_1
  \le
  \|v\|_1
  \left(\prod_{r=0}^{j-1}(1+rh)C_2\right)
  \exp\left(\mu_1\sum_{r=0}^{j}(1+r)\ell_r\right).
\]
If \(\mu_1\ge0\), then
\[
  \|\widetilde q_j\|_1
  \le
  \|v\|_1
  \left(\prod_{r=0}^{j-1}(1+rh)C_2\right)
  \exp\left(\mu_1\sum_{r=0}^{j}(1+rh)\ell_r\right).
\]
Since
\[
  \prod_{r=0}^{j-1}(1+rh)C_2
  =
  P_jC_2^j,
\]
this accounts for all \(A_2\)-factors.
We now simplify the exponential factors. From the definitions of the
\(\ell_r\)'s,
\[
  \sum_{r=0}^{j}(1+r)\ell_r
  =
  s+s_1+\cdots+s_j,
\]
and
\[
  \sum_{r=0}^{j}(1+rh)\ell_r
  =
  s+h(s_1+\cdots+s_j).
\]
Therefore, by the definition of \(\sigma\), both sign cases are written
uniformly as
\[
  \|\widetilde q_j\|_1
  \le
  \|v\|_1
  P_jC_2^j
  e^{\mu_1s}e^{\sigma\mu_1(s_1+\cdots+s_j)} .
\]
Finally, using the terminal \(\ell^1\)-\(\ell^\infty\) estimate,
\[
  |\widetilde q_j w|
  \le
  \|\widetilde q_j\|_1\|w\|_\infty,
\]
we obtain the pointwise bound on the fixed \(j\)-th Dyson integrand:
\[
  \left|
  v^Te^{A_1(s-s_j)}
  A_2e^{A_1(s_j-s_{j-1})}
  \cdots
  A_2e^{A_1s_1}w
  \right|
  \le
  \|v\|_1\|w\|_\infty
  P_jC_2^j
  e^{\mu_1s}e^{\sigma\mu_1(s_1+\cdots+s_j)} .
\]
Integrating this estimate over \(\Delta_j(s)\) gives
\[
  \int_{\Delta_j(s)}
  e^{\sigma\mu_1(s_1+\cdots+s_j)}
  \,ds_1\cdots ds_j .
\]
The integrand is symmetric in \(s_1,\ldots,s_j\), and the ordered simplices
partition \([0,s]^j\) up to measure-zero boundaries. Hence
\[
  \int_{\Delta_j(s)}
  e^{\sigma\mu_1(s_1+\cdots+s_j)}
  \,ds_1\cdots ds_j
  =
  \frac1{j!}
  \int_{[0,s]^j}
  e^{\sigma\mu_1(u_1+\cdots+u_j)}
  \,du_1\cdots du_j .
\]
The cube integral factors:
\[
  \int_{[0,s]^j}
  e^{\sigma\mu_1(u_1+\cdots+u_j)}
  \,du_1\cdots du_j
  =
  \left(\int_0^s e^{\sigma\mu_1u}\,du\right)^j
  =
  \Theta(s)^j .
\]
Therefore
\[
  \int_{\Delta_j(s)}
  e^{\sigma\mu_1(s_1+\cdots+s_j)}
  \,ds_1\cdots ds_j
  =
  \frac{\Theta(s)^j}{j!}.
\]
Summing over \(j=k_D,\ldots,D-1\), we obtain
\[
  |v^Te^{As}w|
  \le
  \|v\|_1\|w\|_\infty e^{\mu_1s}
  \sum_{j=k_D}^{D-1}
  P_jC_2^j\frac{\Theta(s)^j}{j!}.
\]
This proves the claim.
\end{proof}

\ifmai
\begin{proof}
By Lemma~\ref{lem:dyson},
\[
  v^Te^{As}w
  =\sum_{j=0}^{\infty}
    \int_{\Delta_j(s)}
    v^Te^{A_1(s-s_j)}A_2e^{A_1(s_j-s_{j-1})}
    \cdots A_2e^{A_1s_1}w
    \,ds_1\cdots ds_j.
\]

First consider the degree support. The row vector \(v^T\) is supported on
monomials of degree \(1\). Each factor \(e^{A_1\theta}\) does not mix total
homogeneous degrees. Each occurrence of \(A_2\) can increase degree by at most
\(h\). Hence after \(j\) occurrences of \(A_2\), the row support has maximum
possible degree at most \(1+jh\). If \(j<k_D\), then
\[
  1+jh<D-h+1,
\]
so the row support is disjoint from the support of \(w\). Therefore all terms
with \(j<k_D\) vanish.

Moreover, each occurrence of \(A_2\), when nonzero, increases degree by at
least \(1\). Hence after \(j\) occurrences of \(A_2\), the row support has
minimum possible degree at least \(1+j\). If \(j\ge D\), then this minimum
exceeds \(D\), and the corresponding row vector in the truncated degree-\(D\)
space is zero. Thus only indices \(k_D\le j\le D-1\) contribute.

Fix such a \(j\), and fix \((s_1,\ldots,s_j)\in\Delta_j(s)\). Before the
\(r\)-th occurrence of \(A_2\), with \(r=0,\ldots,j-1\), the row support has
maximum degree at most \(1+rh\). By Lemma~\ref{lem:degree-structure}, the
corresponding \(A_2\)-factor has \(\ell^1\)-operator contribution at most
\[
  (1+rh)C_2.
\]
The product of all \(A_2\)-bounds is therefore
\[
  P_jC_2^j.
\]

It remains to bound the \(A_1\)-semigroup factors. Let
\[
  \ell_0:=s-s_j,
  \quad
  \ell_r:=s_{j-r+1}-s_{j-r}\quad(1\le r\le j-1),
  \quad
  \ell_j:=s_1.
\]
Equivalently, the interval of length \(\ell_r\) is the interval after exactly
\(r\) occurrences of \(A_2\) have been applied to the row vector.

If \(\mu_1<0\), then after \(r\) occurrences of \(A_2\) the row support is
supported only on degrees at least \(1+r\). By
Lemma~\ref{lem:degree-local-a1}, the product of \(A_1\)-semigroup bounds is at
most
\[
  \exp\left(\mu_1\sum_{r=0}^j(1+r)\ell_r\right).
\]
A direct telescoping computation gives
\[
  \sum_{r=0}^j(1+r)\ell_r=s+s_1+\cdots+s_j.
\]
Thus, in the case \(\mu_1<0\), the exponential contribution is bounded by
\[
  e^{\mu_1s}e^{\mu_1(s_1+\cdots+s_j)}.
\]

If \(\mu_1\ge0\), then after \(r\) occurrences of \(A_2\) the row support is
supported only on degrees at most \(1+rh\). By
Lemma~\ref{lem:degree-local-a1}, the product of \(A_1\)-semigroup bounds is at
most
\[
  \exp\left(\mu_1\sum_{r=0}^j(1+rh)\ell_r\right).
\]
Again by telescoping,
\[
  \sum_{r=0}^j(1+rh)\ell_r=s+h(s_1+\cdots+s_j).
\]
Hence, in the case \(\mu_1\ge0\), the exponential contribution is bounded by
\[
  e^{\mu_1s}e^{h\mu_1(s_1+\cdots+s_j)}.
\]

By the definition of \(\sigma\), both cases can be written as
\[
  e^{\mu_1s}e^{\sigma\mu_1(s_1+\cdots+s_j)}.
\]
Therefore the absolute value of the fixed \(j\)-th integrand is bounded by
\[
  \|v\|_1\|w\|_\infty P_jC_2^j
  e^{\mu_1s}e^{\sigma\mu_1(s_1+\cdots+s_j)}.
\]
Integrating over \(\Delta_j(s)\), and using symmetry of the product over the
cube \([0,s]^j\), yields
\[
  \int_{\Delta_j(s)} e^{\sigma\mu_1(s_1+\cdots+s_j)}
  \,ds_1\cdots ds_j
  =\frac1{j!}\left(\int_0^s e^{\sigma\mu_1u}\,du\right)^j
  =\frac{\Theta(s)^j}{j!}.
\]
Summing over \(j=k_D,\ldots,D-1\) proves the claim.
\end{proof}
\fi

\subsection{General bound}
\ifmai
\begin{theorem}[Unified refined finite bound]
\label{thm:finite}
For every \(D\ge1\) and every \(t\in[0,T]\),
\[
\boxed{
  \epsilon_g^D(t;x_0)
  \le
  \|v\|_1DC_2R^{D+1}R_+^{h-1}
  \sum_{j=k_D}^{D-1}P_jC_2^j
  \frac{\Theta(t)^{j+1}}{(j+1)!}.
}
\]
Consequently,
\[
\boxed{
  \sup_{t\in[0,T]}\epsilon_g^D(t;x_0)
  \le
  \|v\|_1DC_2R^{D+1}R_+^{h-1}
  \sum_{j=k_D}^{D-1}P_jC_2^j
  \frac{\Theta(T)^{j+1}}{(j+1)!}.
}
\]
\end{theorem}
\fi
\begin{proof_of}{Theorem \ref{thm:finite}}
By Lemma~\ref{lem:voc}, with \(s=t-\tau\),
\[
  \epsilon_D(t;x_0)
  \le \int_0^t |v^Te^{A(t-\tau)}w(\tau)|\,d\tau.
\]
By Lemma~\ref{lem:degree-structure},
\[
  \|w(\tau)\|_\infty\le DC_2R^{D+1}R_+^{h-1}.
\]
By Lemma~\ref{lem:refined-tail},
\[
  |v^Te^{As}w(\tau)|
  \le
  \|v\|_1DC_2R^{D+1}R_+^{h-1}
  e^{\mu_1s}
  \sum_{j=k_D}^{D-1}P_jC_2^j\frac{\Theta(s)^j}{j!}.
\]
Therefore
\[
  \epsilon_D(t;x_0)
  \le
  \|v\|_1DC_2R^{D+1}R_+^{h-1}
  \sum_{j=k_D}^{D-1}P_jC_2^j\frac1{j!}
  \int_0^t e^{\mu_1s}\Theta(s)^j\,ds.
\]
If \(\mu_1<0\), then \(\sigma=1\), and
\[
  \Theta'(s)=e^{\mu_1s}.
\]
If \(\mu_1\ge0\), then \(\sigma=h\), and since \(h\ge1\),
\[
  e^{\mu_1s}\le e^{h\mu_1s}=\Theta'(s).
\]
Thus, in all cases,
\[
  \int_0^t e^{\mu_1s}\Theta(s)^j\,ds
  \le
  \int_0^t \Theta'(s)\Theta(s)^j\,ds
  =\frac{\Theta(t)^{j+1}}{j+1}.
\]
Substituting this estimate gives
\[
  \epsilon_g^D(t;x_0)
  \le
  \|v\|_1DC_2R^{D+1}R_+^{h-1}
  \sum_{j=k_D}^{D-1}P_jC_2^j
  \frac{\Theta(t)^{j+1}}{(j+1)!}.
\]
Since \(\Theta\) is increasing on \([0,T]\), the uniform bound follows.
\end{proof_of}

\subsection{Geometric corollaries}
Define
\[
  \rho:=hC_2\Theta(T).
\]

\vspace*{.2cm}
\ifmai
\begin{corollary}[First-tail geometric regime]
\label{cor:first-tail}
Assume
\[
  \rho=hC_2\Theta(T)<1.
\]
Then, for all \(D\ge1\),
\[
\boxed{
  \sup_{t\in[0,T]}\epsilon_g^D(t;x_0)
  \le K_1\,\gamma_1^{\,k_D},
}
\]
where
\[
\boxed{
  \gamma_1:=hC_2R^h\Theta(T),
}
\]
and
\[
\boxed{
  K_1:=\frac{\|v\|_1hC_2\Theta(T)R^2R_+^{2h-2}}{1-\rho}.
}
\]
In particular, if \(\gamma_1<1\), then
\[
  \sup_{t\in[0,T]}\epsilon_g^D(t;x_0)\to0
  \qquad\text{as }D\to+\infty.
\]
\end{corollary}
\fi
\begin{proof_of}{Corollary \ref{cor:first-tail}}
For every \(j\ge0\),
\[
  P_j=\prod_{r=0}^{j-1}(1+rh)
  \le\prod_{r=0}^{j-1}h(r+1)=h^jj!.
\]
Using Theorem~\ref{thm:finite},
\[
  \sup_{t\in[0,T]}\epsilon_D(t;x_0)
  \le
  \|v\|_1DC_2R^{D+1}R_+^{h-1}
  \Theta(T)
  \sum_{j=k_D}^{D-1}
  \frac{(hC_2\Theta(T))^j}{j+1}.
\]
Since \(k_D+1=\lceil D/h\rceil\), for every \(j\ge k_D\),
\[
  \frac{D}{j+1}\le\frac{D}{k_D+1}\le h.
\]
Therefore
\[
  \sup_{t\in[0,T]}\epsilon_D(t;x_0)
  \le
  \|v\|_1hC_2\Theta(T)R^{D+1}R_+^{h-1}
  \sum_{j=k_D}^{D-1}\rho^j.
\]
If \(\rho:=\rho(T)<1\), then
\[
  \sum_{j=k_D}^{D-1}\rho^j\le\frac{\rho^{k_D}}{1-\rho}.
\]
It remains to absorb the factor \(R^{D+1}\). Since
\[
  D-1-hk_D\in\{0,1,\ldots,h-1\},
\]
we have
\[
  R^{D+1}=R^2R^{hk_D}R^{D-1-hk_D}
  \le R^2R_+^{h-1}(R^h)^{k_D}.
\]
Thus
\[
  R^{D+1}R_+^{h-1}
  \le R^2R_+^{2h-2}(R^h)^{k_D}.
\]
Combining the estimates gives
\[
  \sup_{t\in[0,T]}\epsilon_g^D(t;x_0)
  \le
  \frac{\|v\|_1hC_2\Theta(T)R^2R_+^{2h-2}}{1-\rho}
  \left(R^h\rho\right)^{k_D}.
\]
Since \(R^h\rho=hC_2R^h\Theta(T)=\gamma_1\), the claim follows.
\end{proof_of}

\vspace*{.2cm}
\vspace*{.2cm}
\ifmai
\begin{corollary}[Last-tail geometric regime]
\label{cor:last-tail}
Assume
\[
  \rho=hC_2\Theta(T)>1
  \qquad\text{and}\qquad
  R\rho<1.
\]
Then, for all \(D\ge1\),
\[
\boxed{
  \sup_{t\in[0,T]}\epsilon_g^D(t;x_0)
  \le K_2\,\gamma_2^{\,k_D},
}
\]
where
\[
\boxed{
  \gamma_2:=(R\rho)^h,
}
\]
and
\[
\boxed{
  K_2:=\frac{\|v\|_1hC_2\Theta(T)RR_+^{h-1}}{\rho-1}.
}
\]
Consequently,
\[
  \sup_{t\in[0,T]}\epsilon_g^D(t;x_0)\to0
  \qquad\text{as }D\to+\infty.
\]
\end{corollary}
\fi
\begin{proof_of}{Corollary \ref{cor:last-tail}}
The proof begins as in Corollary~\ref{cor:first-tail}. From Theorem~\ref{thm:finite},
\(P_j\le h^jj!\), and \(D/(j+1)\le h\) for \(j\ge k_D\), we obtain
\[
  \sup_{t\in[0,T]}\epsilon_D(t;x_0)
  \le
  \|v\|_1hC_2\Theta(T)R^{D+1}R_+^{h-1}
  \sum_{j=k_D}^{D-1}\rho^j.
\]
If \(\rho:=\rho(T)>1\), then
\[
  \sum_{j=k_D}^{D-1}\rho^j\le\frac{\rho^D}{\rho-1}.
\]
Therefore
\[
  \sup_{t\in[0,T]}\epsilon_D(t;x_0)
  \le
  \frac{\|v\|_1hC_2\Theta(T)R_+^{h-1}}{\rho-1}
  R^{D+1}\rho^D.
\]
Since \(R^{D+1}\rho^D=R(R\rho)^D\),
\[
  \sup_{t\in[0,T]}\epsilon_D(t;x_0)
  \le
  \frac{\|v\|_1hC_2\Theta(T)RR_+^{h-1}}{\rho-1}
  (R\rho)^D.
\]
Moreover, \(D\ge hk_D\). Since \(0<R\rho<1\),
\[
  (R\rho)^D\le (R\rho)^{hk_D}.
\]
Thus
\[
  \sup_{t\in[0,T]}\epsilon_D(t;x_0)
  \le K_2\,\gamma_2^{\,k_D},
\]
with \(K_2\) and \(\gamma_2\) as stated. Since \(R\rho<1\), one has
\(\gamma_2<1\), and convergence follows.
\end{proof_of}

\ifmai
\section{Conclusion and Further Work}\label{sec:conclusion}
We have developed refined error bounds for finite Carleman truncations of
polynomial ordinary differential equations. The bounds are obtained directly in
the original monomial basis and for a prescribed observable, such as a state
coordinate. The key idea is to exploit the degree structure of the Carleman
matrix: the linear part preserves total degree, while the nonlinear part raises
degree. A Dyson--Duhamel expansion makes this separation explicit and allows us
to track how truncation errors propagate from the discarded high-degree tail
back to the observable.

The resulting estimates are degree-aware and retain information on the linear
dynamics through logarithmic norms. In particular, when the linear part has
negative logarithmic norm, this contractivity enters the convergence factor
quantitatively. The examples show that this can reduce conservativeness with
respect to more global norm-based bounds, while avoiding a preliminary quadratic
reduction and keeping the analysis tied to the chosen observable.

Several directions remain for future work. First, the bounds should be integrated
into a reachability implementation, such as CKR~\cite{full}, where the
Carleman truncation matrix can be reused across time steps and only the local
trajectory bound and error enclosure need to be updated. Second, the present
truncation estimates should be combined with Krylov projection-error bounds, in
order to obtain a unified certificate for reduced Carleman--Krylov models.
Finally, it would be useful to extend the degree-local analysis to richer
observables, to broader analytic systems, and to bounds that use set-dependent
or time-varying trajectory radii instead of a single uniform radius.
\fi

\section{Conclusion and Further Work}\label{sec:conclusion}
We have developed refined error bounds for finite Carleman truncations of
polynomial ordinary differential equations. The bounds are obtained directly in
the original monomial basis and for a prescribed observable, such as a state
coordinate. The key idea is to exploit the degree structure of the Carleman
matrix: the linear part preserves total degree, while the nonlinear part raises
degree. A Dyson--Duhamel expansion makes this separation explicit and allows us
to track how truncation errors propagate from the discarded high-degree tail
back to the observable. The resulting estimates are degree-aware and retain information on the linear
dynamics through logarithmic norms.

Several directions remain for future work. First, the bounds should be integrated
into a reachability implementation, such as CKR~\cite{full}, where the
Carleman truncation matrix can be reused across time steps and only the local
trajectory bound and error enclosure need to be updated.
Second, the same questions arise in recent quantum algorithms for nonlinear differential
equations based on Carleman linearization
\cite{LiuEtAl2021,Krovi2023,CostaSchleichMoralesBerry2025,WuWangLi2025}.
Exploring whether degree-local truncation bounds can help relax strong
dissipativity or contractivity assumptions in that setting is another direction
for future work.
\bibliographystyle{plain}
\bibliography{cas-refs}
\end{document}
\ifmai
\section{Improved geometric corollaries}
\begin{corollary}[Exact geometric-sum bound]
\label{cor:exact-geometric-sum}
Let
\[
  h:=p-1,
  \qquad
  \rho:=hC_2\Theta(T).
\]
Then, for every \(D\ge1\),
\[
  \sup_{t\in[0,T]}\epsilon_g^D(t;x_0)
  \le
  \|v\|_1 hC_2\Theta(T)R^{D+1}R_+^{h-1}
  S_{k_D,D}(\rho),
\]
where
\[
  S_{k,D}(\rho):=\sum_{j=k}^{D-1}\rho^j.
\]
Equivalently,
\[
  S_{k,D}(\rho)
  =
  \begin{cases}
  \dfrac{\rho^k-\rho^D}{1-\rho}, & \rho\ne1,\\[1.2em]
  D-k, & \rho=1.
  \end{cases}
\]
\end{corollary}

\begin{proof}
Starting from the finite refined bound,
\[
  \epsilon_g^D(t;x_0)
  \le
  \|v\|_1
  D C_2 R^{D+1}R_+^{h-1}
  \sum_{j=k_D}^{D-1}
  P_jC_2^j\frac{\Theta(T)^{j+1}}{(j+1)!},
\]
we use
\[
  P_j=\prod_{r=0}^{j-1}(1+rh)
  \le
  h^j j!.
\]
Therefore
\[
  P_jC_2^j\frac{\Theta(T)^{j+1}}{(j+1)!}
  \le
  \Theta(T)\frac{(hC_2\Theta(T))^j}{j+1}
  =
  \Theta(T)\frac{\rho^j}{j+1}.
\]
Since \(j\ge k_D\) and \(k_D+1=\lceil D/h\rceil\), we have
\[
  \frac{D}{j+1}
  \le
  \frac{D}{k_D+1}
  \le h.
\]
Hence
\[
  D C_2
  P_jC_2^j\frac{\Theta(T)^{j+1}}{(j+1)!}
  \le
  h C_2\Theta(T)\rho^j.
\]
Substituting into the finite refined bound yields
\[
  \epsilon_g^D(t;x_0)
  \le
  \|v\|_1 hC_2\Theta(T)R^{D+1}R_+^{h-1}
  \sum_{j=k_D}^{D-1}\rho^j.
\]
Taking the supremum over \(t\in[0,T]\) gives the claim. The displayed closed
form for \(S_{k,D}(\rho)\) is the finite geometric-series identity.
\end{proof}
\fi
\end{document}